\documentclass[12pt]{amsart}
\usepackage{tikz}
\usepackage{amssymb}
\usepackage{amsmath}
\usepackage{mathrsfs}
\usepackage{mathbbol}
\usepackage{amsfonts}
\usepackage{amssymb,amsmath}
\usepackage{graphicx}
\graphicspath{{./images/}}
\def\R {\mathbb{R}}

\def\d{{\,\rm d}}

\newtheorem{proposition}{Proposition}[section]
\newtheorem{theorem}[proposition]{Theorem}
\newtheorem{corollary}[proposition]{Corollary}
\newtheorem{lemma}[proposition]{Lemma}
\theoremstyle{definition}

\newtheorem{remark}[proposition]{Remark}
\numberwithin{equation}{section}

\def \no#1#2#3 {{\bf #1} (#3), #2.}
\def \eds#1#2#3 {#1, #2, #3.}

\title[Analyticity and observability]
{ \bf
 Analyticity and observability for fractional heat equation on $\R^n$}

\author[M. Wang and C. Zhang]
{Ming Wang and Can Zhang}

\address{Ming Wang
\newline\indent
School of Mathematics and Physics, China University of Geosciences
\newline\indent
Wuhan, 430074, P.R. China
}
\email{mwang@cug.edu.cn}

\address{Can Zhang
\newline\indent
School of Mathematics and Statistics, Wuhan University; Computational
Science Hubei Key Laboratory, Wuhan University, Wuhan 430072, P.R. China
}
\email{canzhang@whu.edu.cn}

\begin{document}
\begin{abstract}
In this paper, we study quantitative spatial analytic bounds and unique continuation inequalities of solutions for  fractional heat equations with an analytic lower order term on the whole space. At first, we show that the solution has a uniform positive analytic radius for all time, and the solution enjoys a log-type ultra-analytic bound if the coefficient is ultra-analytic. Second, we prove a H\"{o}lder type interpolation inequality on a thick set, with an explicit dependence  on the analytic radius of coefficient. Finally, by the telescoping series method, we establish an observability inequality from a thick set. As a byproduct of the proof, we obtain observability inequalities in  weighted spaces  from a thick set  for the classical heat equation with a lower order term.
\end{abstract}



\maketitle


\section{Introduction}
 It has been recently  proved in \cite{WWZZ,EV18} that the observability inequality
\begin{align}\label{equ-1}
\forall T>0, \exists C=C(n,T,E)>0 \mbox{ so that } \int_{\R^n}|u(T,x)|^2\d x\leq C\int_0^T\int_E|u(t,x)|^2\d x \d t
\end{align}
and the interpolation inequality
\begin{align}\label{equ-2}
\begin{split}
&\forall T>0, \forall \theta\in(0,1), \exists C=C(n,T,\theta,E)>0 \mbox{ so that }\\
& \int_{\R^n}|u(T,x)|^2\d x\leq C\left(\int_E|u(T,x)|^2\d x \right)^\theta\left(\int_{\R^n}|u(0,x)|^2\d x\right)^{1-\theta}
\end{split}
\end{align}
hold for all solutions of the heat equation
\begin{align}\label{equ-3}
\partial_tu-\Delta u=0\;\;\;\text{in}\;\;\mathbb R^+\times\mathbb R^n, \quad u(0,x)\in L^2(\R^n),
\end{align}
if and only if $E\subset \R^n$ is thick, namely there exists $L>0$ so that
$$
\inf_{x\in \R^n}|E\bigcap Q_{L}(x)|>0.
$$
Here $|\cdot|$ denotes the Lebesgue measure, $Q_L(x)$ stands for the cube in $\R^n$ centered at $x$ with side length $L>0$.

The thick set has been first introduced  in the study of the uncertainty principle of Fourier transform. In fact, the Logvinenko-Sereda theorem (see e.g. \cite{HJ,Kov})   says that
\begin{align}\label{equ-4}
\int_{\R^n}|f(x)|^2\d x \leq Ce^{CN}\int_E|f(x)|^2\d x, \quad \forall N>0, \forall f\in L^2(\R^n), \mbox{ supp } \widehat{f} \subset B_N(0)
\end{align}
holds for some constant $C=C(n,E)>0$ if and only if $E$ is thick. Here $\widehat{f}$ denotes the Fourier transform of $f$, the ball $B_N(x_0)=\{x\in \R^n: |x-x_0|\leq N\}$. The inequality \eqref{equ-4} is called spectral inequality \cite{LeJDE}, which plays an important role in the the Lebeau-Robbiano strategy to establish observability inequalities.

 In fact, let $\mathbb{H}$ be a self-adjoint operator so that $-\mathbb{H}$ generates a $C_0$ semigroup $\{e^{-t\mathbb{H}}\}_{t\geq 0}$ in $L^2(\R^n)$, and let $\{\pi_N\}_{N\geq 1}$ be a family of orthogonal projection operators on $L^2(\R^n)$. Then we recall the Lebeau-Robbiano strategy \cite{Lebeau95,Miller10,TT11,Miller12,BP18,Bea20}: If there exist $b>a>0$ and $C>0$ so that the spectral inequality
\begin{align}\label{equ-5}
\|\pi_Nf\|_{L^2(\R^n)}\leq Ce^{CN^a}\|\pi_Nf\|_{L^2(E)}
\end{align}
and the dissipative inequality
\begin{align}\label{equ-6}
\|(1-\pi_N)e^{-tH}f\|_{L^2(\R^n)}\leq Ce^{-CtN^b}\|(1-\pi_N)f\|_{L^2(\R^n)}, \quad \forall t>0
\end{align}
hold for all $N\geq 1$ and $f\in L^2(\R^n)$, then  the following observability inequality holds
\begin{align*}%
\begin{split}
&\forall T>0, \exists C=C(n,T,E)>0 \mbox{ so that } \\
&\int_{\R^n}|e^{-TH}f(x)|^2\d x\leq C\int_0^T\int_E|e^{-tH}f(x)|^2\d x \d t, \quad \forall f\in L^2(\R^n).
\end{split}
\end{align*}

If we let $\pi_N$ be the Fourier projection operator defined by
\begin{align}\label{equ-7}
\widehat{\pi_Nf}(\xi)=\chi_{|\xi|\leq N}\widehat{f}(\xi),
\end{align}
where $\chi_{|\xi|\leq N}$ is the characteristic function of the set $\{\xi\in\R^n: |\xi|\leq N\}$, then the inequality \eqref{equ-4} can be rewritten as
\begin{align}\label{equ-8}
\|\pi_Nf\|_{L^2(\R^n)} \leq Ce^{CN}\|\pi_Nf\|_{L^2(E)}, \quad \forall N>0,  f\in L^2(\R^n).
\end{align}
Moreover, for every $s>0$, let $\Lambda^s=(-\Delta)^{\frac{s}{2}}$ be the fractional Laplacian defined by the Fourier transform
\begin{align}\label{equ-8.5}
\widehat{\Lambda^s f}  = |\xi|^s\widehat{f}(\xi).
\end{align}
By the  Plancherel theorem,  we have
\begin{align}\label{equ-9}
\|(1-\pi_N)e^{-t\Lambda^s}f\|_{L^2(\R^n)}\leq Ce^{-tN^s}\|(1-\pi_N)f\|_{L^2(\R^n)}, \quad \forall t, N>0,   f\in L^2(\R^n).
\end{align}
According to the above Lebeau-Robbiano strategy, if $s>1$ and $E$ is a thick set, then the observability inequality \eqref{equ-1} holds for all solutions of
\begin{align}\label{equ-10}
\partial_t u+\Lambda^s u=0\;\;\;\text{in}\;\;\mathbb R^+\times\mathbb R^n, \quad u(0,x)=u_0\in L^2(\R^n).
\end{align}

In particular, letting $s=2$, this recovers the observability \eqref{equ-1} for the heat equation \eqref{equ-3}. We note that the restriction $s>1$, comes from the assumption $b>a$ in \eqref{equ-5}-\eqref{equ-6}, is essential\footnote{But for the exponential stabilization of the fractional heat equation,  this restriction can be removed, see \cite[Lemma 2.2]{HWW}.}. In fact, if $0<s\leq 1$, then the equation \eqref{equ-10} is not null controllable on a thick set $E$ (say, $E$ is the complement  of a nonempty open set, see \cite{Koenig20,Lissy20}).

Based on Carleman estimates, Lebeau and Moyano \cite{Lebeau} have proved a spectral inequality for the Schr\"{o}dinger operator $H_{g,V}=\Delta_g+V(x)$ in $\R^n$, where $\Delta_g$ is the Laplace-Beltrami operator with respect to an analytic metric $g$, $V(x)$ is an analytic function vanishes at infinity. Precisely, if $E$ is a thick set, then there exists $C>0$ so that
\begin{align}\label{equ-11}
\|\pi_N f\|_{L^2(\R^n,\sqrt{\det g}\d x)}\leq Ce^{C\sqrt{N}}\|\pi_N f\|_{L^2(E,\sqrt{\det g}\d x)}, \quad \forall N>0, f\in L^2(\R^n),
\end{align}
where $\pi_N$ is a spectral projection to the low frequency (see \cite{Lebeau} for a precise definition). Clearly, the inequality \eqref{equ-11}, by the Lebeau-Robbiano strategy, implies that the observability inequality \eqref{equ-1} holds for all solutions of the heat equation with real analytic potentials $V(x)$:
$$
\partial_t u-\Delta u=V(x)u\;\;\;\text{in}\;\;\mathbb R^+\times\mathbb R^n, \quad u(0,x)\in L^2(\R^n).
$$

Motivated by these works, we wonder that, to what extent  the above mentioned results can be extended to the fractional heat equation with a space-time potential
\begin{align}\label{frac-heat}
\partial_t u+\Lambda^s u=a(t,x)u\;\;\;\text{in}\;\;\mathbb R^+\times\mathbb R^n, \quad u(0,x)=u_0\in L^2(\R^n),
\end{align}
where $s>1$, $\Lambda^s$ is defined by \eqref{equ-8.5}. In particular, whether the inequalities \eqref{equ-1}-\eqref{equ-2} hold for all solutions of \eqref{frac-heat} with analytic potential $a(t,x)$?

We first note that, due to the nonlocal property of $\Lambda^s$ and the time dependence of $a(t,x)$, it is not clear that how to adapt the approach of Lebeau and Moyano in \cite{Lebeau} to the equation \eqref{frac-heat}. Moreover, if one uses the spectral inequality \eqref{equ-8}, with $\pi_N$ being the Fourier projection defined by \eqref{equ-7}, then according to an adapted Lebeau-Robbiano strategy \cite{Bea20}, the observability \eqref{equ-1} reduces to the following dissipative estimate
\begin{align}\label{equ-12}
\|(1-\pi_N)u(t,\cdot)\|_{L^2(\R^n)}\leq C(t)e^{-CtN^\gamma}\|u_0\|_{L^2(\R^n)}, \quad \forall t, N>0,   u_0\in L^2(\R^n),
\end{align}
for some $\gamma>1$ (corresponding to $b>a$ in \eqref{equ-5}-\eqref{equ-6}), where $u(t,x)$ is the solution of \eqref{frac-heat}. However, the estimate  \eqref{equ-12} is equivalent to
\begin{align}\label{equ-13}
\|e^{ct|\xi|^\gamma}\widehat{u}(t,\xi)\|_{L^2_\xi(\R^n)}\leq C(t)\|u_0\|_{L^2(\R^n)}, \quad \forall t>0,     u_0\in L^2(\R^n),
\end{align}
which, to the best of our knowledge, is still open even if $a(t,x)$ satisfying that
$$
\sup_{t>0,x\in\R^n}|\partial_x^\alpha a(t,x)|\leq 1, \quad \forall \alpha\in \mathbb{N}^n.
$$
In fact, the bound \eqref{equ-13} implies that (see e.g. \cite[Lemma 3.3, p.131]{WWZZ})
\begin{align}\label{equ-818-1}
\|\partial_x^\alpha u(t,\cdot)\|_{L^\infty(\R^n)}\leq C^{|\alpha|}(t)(\alpha!)^{\frac{1}{\gamma}},\quad \forall \alpha\in \mathbb{N}^n,
\end{align}
which, usually called ultra-analytic estimate, is stronger than the usual analytic bound
\begin{align}\label{equ-818-2}
\|\partial_x^\alpha u(t,\cdot)\|_{L^\infty(\R^n)}\leq C^{|\alpha|}(t)\alpha!, \quad \forall \alpha\in \mathbb{N}^n,
\end{align}
since $s>1$. The bound of the form \eqref{equ-818-2} has been studied for Navier-Stokes equations, see e.g. \cite{OT,Dong08,Dong09,Her}. We also mention that the time analyticity has been proved in \cite{Dong20}. But little is known on the bound \eqref{equ-818-1} for heat equations.

Thus, we shall first  study  analytical bounds toward to \eqref{equ-13} for the solution of \eqref{frac-heat}. To state our main results, we make two assumptions on $a(t,x)$.
\begin{description}
  \item [(A1) Analyticity] There exist constants $C,R>0$ so that
  $$
  \sup_{t>0,x\in \R^n}|\partial_x^\alpha a(t,x)|\leq C\frac{\alpha!}{R^{|\alpha|}}, \quad \forall \alpha\in \mathbb{N}^n.
  $$
  \item [(A2) Ultra-analyticity] There exist  constants $C,M>0,\kappa\in[0,1)$ so that
  $$
  \sup_{t>0,x\in \R^n}|\partial_x^{\alpha}a(t,x)|\leq CM^{|\alpha|}(\alpha!)^{\kappa}, \quad \forall \alpha\in \mathbb{N}^n.
  $$
\end{description}
\begin{theorem}\label{thm-ana}
Let $s>1$ and let $u(t,x)$ be the solution of \eqref{frac-heat}.
\begin{itemize}
  \item[(i)] Assume $\bf(A1)$ holds. Then there exist constants $c>0,C>0$ so that for all $t>0$
\begin{align}\label{ana-b-1}
\left\|\widehat{u}(t,\cdot)e^{cR|\xi|}\right\|_{L_\xi^2(\R^n)} \leq \exp\left\{C\Big[1+\left(t^{-1}R^s\right)^{\frac{1}{s-1}}+t\sup_{t>0}\|a(t,\cdot)\|_{A^{\frac{R}{2}}}\Big]\right\} \|u_0\|_{L^2(\R^n)},
\end{align}
where the norm $\|\cdot\|_{A^R}$ is defined by \eqref{equ-def-A-norm}.
  \item[(ii)]Assume $\bf(A2)$ holds. Then there exist constants $c>0,C>0$ so that for all $t>0$
\begin{align}\label{ana-b-2}
\|e^{c|\xi|(\log(e+|\xi|))^{1-\kappa}}\widehat{u}(t,\xi)\|_{L^2_{\xi}(\R^n)} \leq C e^{C(t^{-\frac{1}{s-1}}+t)}\|u_0\|_{L^2(\R^n)}.
\end{align}
\end{itemize}
\end{theorem}

The bound \eqref{ana-b-1} shows that $\|\widehat{u}(t,\cdot)e^{cR|\xi|}\|_{L_\xi^2(\R^n)}$ is finite for every $t>0$. This, according to the Paley-Wiener theorem \cite[Theorem IX.13, p.18]{Simon}, implies that the solution $u(t,\cdot)$ can be extended to an analytic function $u(t,z)$ in the strip
$$
S_\sigma=\left\{z\in \mathbb{C}^n: |\mbox{Im}z|<\sigma\right\}
$$
with $\sigma=cR$. In particular, the solution has a fixed analytic radius at every time if the lower order term is analytic. Results in similar manner has been proved in \cite{Zhangcan17} for $2m$ ($m$ is an integer) order parabolic equation on bounded domains. The proof in \cite{Zhangcan17} relies on  Schauder estimates and a delicate iteration argument, while \eqref{ana-b-1} is proved by some tools in Fourier analysis.

The main novelty of \eqref{ana-b-1} lies in that it gives quantitative information on the analytic radius of the solution and upper bound constant in terms of $R$, the analytic radius of the coefficient $a(t,\cdot)$. This is the key ingredient in the proof of the estimate \eqref{ana-b-2}.

 Since for every $R>0$,
$$
\|e^{R|\xi|}\widehat{u}(t,\xi)\|_{L^2_{\xi}(\R^n)}\leq C(R)\|e^{c|\xi|(\log(e+|\xi|))^{1-\kappa}}\widehat{u}(t,\xi)\|_{L^2_{\xi}(\R^n)},
$$
it follows from \eqref{ana-b-2} that the solution $u(t,x)$, for every $t>0$, can be extended to an analytic function on the whole $\mathbb{C}^n$.
Though \eqref{ana-b-2} is weaker than the classical ultra-analytic estimate \eqref{equ-13}, it is new for us. We call \eqref{ana-b-2} a log-type  ultra-analytic estimate.

%

With these quantitative analytic bounds in hand, we prove some H\"{o}lder type interpolation inequalities of unique continuation as follows.
\begin{theorem}\label{thm-interpolation}
Let $s>1$, $E$ be a thick set in $\R^n$ and $u(t,x)$ be the   solution  of \eqref{frac-heat}.
\begin{itemize}
  \item[(i)]Assume that $\bf(A1)$ holds for some $R>0$. Then, there exist constants $C=C(n,E,\sigma)>0$ and $C'=C'(n,E)>0$ so that
\begin{align}\label{equ-inter-81-1}
\int_{\R^n}|u(t,x)|^2\d x\leq C_0\left(\int_E|u(t,x)|^2 \d x\right)^\theta\|u_0\|^{2(1-\theta)}_{L^2(\R^n)}
\end{align}
holds for all $\theta\in(0,e^{-C'\max\{1,R^{-1}\}})$, where
$$
C_0=C\exp\left\{C\Big[1+\left(t^{-1}R^s\right)^{\frac{1}{s-1}}+t\sup_{t>0}\|a(t,\cdot)\|_{A^{\frac{R}{2}}}\Big]\right\}.
$$
  \item[(ii)]Assume that $\bf(A2)$ holds. Then, there exist $C>0$ so that for any $\theta\in(0,1)$
\begin{align}\label{equ-inter-81-2}
\int_{\R^n}|u(t,x)|^2\d x \leq C e^{C(t^{-\frac{1}{s-1}}+t)}e^{e^{C\left(\frac{\theta}{1-\theta} \right)^{\frac{1}{1-\kappa}}}} \left(\int_E|u(t,x)|^2 \d x\right)^\theta
\|u_0\|^{2(1-\theta)}_{L^2(\R^n)}.
\end{align}
\end{itemize}
\end{theorem}

\medskip

In order to prove \eqref{equ-inter-81-1}, we shall establish an interpolation   inequality for a function $f$ satisfying $\|e^{R|\xi|}\widehat{f}\|_{L^2(\R^n)}<\infty$ on a thick set. This is a slightly stronger than previous versions in the existing  literature, we refer the interesting reader  to Remark \ref{rem-ana-inter} for the history of this topic.

The proof of \eqref{equ-inter-81-2} relies on the Logvinenko-Sereda theorem and a high-low frequency decomposition. The inequality \eqref{equ-inter-81-2} is stronger than \eqref{equ-inter-81-1} in the sense that it allows the H\"{o}lder exponent $\theta$ close to $1$ arbitrarily. Moreover, it shows that the interpolation inequality \eqref{equ-2} holds at least for ultra-analytic lower order terms.

With regards to  the observability inequality for solutions to \eqref{frac-heat}, we have the following result.

\begin{theorem}\label{thm-ob}
Assume that $s>1$ and $\bf(A1)$ holds,  and $E\subset \R^n$ is a thick set. Then there exists a constant $C>0$ depending only on $n,a$ and $E$ so that for all $T>0$ and all solutions of \eqref{frac-heat},
\begin{align}\label{equ-ob}
\int_{\R^n}|u(T,x)|^2\d x\leq Ce^{C(T+\frac{1}{T^{s-1}})}\int_0^T\int_E|u(t,x)|^2\d x \d t.
\end{align}
\end{theorem}

\medskip

Theorem \ref{thm-ob} generalizes the observability inequalities in \cite{WWZZ,EV18}. As mentioned above, Theorem \ref{thm-ob} does not follows directly from  the Lebeau-Robbiano strategy. But it can be proved with a similar idea. In fact, based on the telescoping method, the observability inequality can be reduced to an interpolation inequality, see Corollary \ref{cor-ob-1}. In this way, we prove  Theorem \ref{thm-ob} by the interpolation inequality \eqref{equ-inter-81-1}. Note that this approach proving observability inequality has been used successfully in \cite{Zhangcan20}.

Based on the techniques developed in this paper, we can establish  the following observability inequalities for heat equations in weighted spaces, which are of independent interest.

\begin{theorem}\label{thm-ob-weight}
Let $\bf(A1)$ hold and  $E\subset \R^n$ be a thick set. Then for every $\delta\in \R$, there exists a constant $C>0$ depending only on $\delta,n,a$ and $E$ so that
\begin{align}\label{equ-ob-weight}
\int_{\R^n}|u(T,x)|^2(1+|x|^2)^\delta\d x\leq Ce^{C(T+\frac{1}{T})}\int_0^T\int_E|u(t,x)|^2(1+|x|^2)^\delta\d x \d t
\end{align}
holds for all $T>0$ and all solutions of
$$
\partial_t u -\Delta u=a(t,x)u,\;\;\;\text{in}\;\;\mathbb R^n\times\mathbb R^+, \quad u(0,x)=u_0\in L^2\big(\R^n,(1+|x|^2)^\delta\d x\big).
$$
\end{theorem}

This paper is mainly devoted to observability estimates for solutions of fractional heat equations with real analytic potentials depending on both space and time variables in the whole space $\R^n$. We refer the reader to, e.g., \cite{Ru71,Micu06,Miller06,Miller10,Miller12,Lissy14,Lissy17,Koenig19} for null controllability results for fractional heat equations on bounded domains.

Throughout the paper, we use $n\geq 1$ to denote the spatial dimension. In some places, we use $A\lesssim B$ to denote $A\leq CB$ for some universal  constant $C>0$. If both $A\lesssim B$ and $B\lesssim A$ hold, then we write $A\sim B$. The Fourier transform is given by
$$
\widehat{f}(\xi)=\int_{\R^n}e^{-ix\cdot\xi}f(x)\d x.
$$
We use $L^2\big(\R^n,(1+|x|^2)^\delta\d x\big)$ to denote  the Hilbert space endowed with the norm
$$
\|f\|_{L^2(\R^n,(1+|x|^2)^\delta\d x)}=\left(\int_{\R^n} |f(x)|^2(1+|x|^2)^\delta\d x  \right)^{\frac{1}{2}}.
$$
It reduces to the usual $L^2(\R^n)$ space if $\delta=0$.

This paper is organized as follows. In Section 2, we prove Theorem \ref{thm-ana}. In Section 3, we first establish some interpolation inequalities for real analytic functions, then prove Theorem \ref{thm-interpolation} with the aid of Theorem \ref{thm-ana}.  Finally, we prove the Theorem \ref{thm-ob} and Theorem \ref{thm-ob-weight} in Section 4.

%

\section{Analytic bounds}

\subsection{Preliminaries}

For every $\sigma>0$, we define the Banach space $G^\sigma=G^\sigma(\R^n)$, consisting of analytic function in $S_\sigma=\left\{z\in \mathbb{C}^n: |\mbox{Im}z|<\sigma\right\}$, endowed with the norm
$$
\|f\|_{G^\sigma}=\sup_{|y|<\sigma}\|f(\cdot+iy)\|_{L^2(\R^n)}.
$$
This kind of analytic functions, according to the Paley-Wiener theorem, is related to the function whose Fourier transform decays exponentially at infinity. The proof of the following lemma is inspired by Problem 76 in \cite[p.132]{Simon}.
\begin{lemma}\label{lem-pre-1}
For all $\sigma>0$ and all $f\in G^\sigma$
\begin{align}\label{equ-82-1}
\|e^{\frac{\sigma}{2}|\xi|}\widehat{f}(\xi)\|_{L_\xi^2(\R^n)}\lesssim\|f\|_{G^\sigma}\lesssim \|e^{\sigma|\xi|}\widehat{f}(\xi)\|_{L_\xi^2(\R^n)}.
\end{align}
\end{lemma}
\begin{proof}
We first claim that
\begin{align}\label{equ-82-2}
\|f\|_{G^\sigma}\sim \sup_{|y|<\sigma}\|e^{y\cdot \xi}\widehat{f}(\xi)\|_{L_\xi^2(\R^n)}.
\end{align}
In fact, by the Fourier inversion,
$$
f(x)=(2\pi)^{-n}\int_{\R^n}e^{ix\cdot\xi}\widehat{f}(\xi)\d \xi, \quad \forall x\in \R^n.
$$
In particular, replacing $x$ by $x+iy$, we find
$$
f(x+iy)=(2\pi)^{-n}\int_{\R^n}e^{ix\cdot\xi}e^{-y\cdot\xi}\widehat{f}(\xi)\d \xi, \quad \forall x\in \R^n
$$
for every $|y|<\sigma$. By the Plancherel theorem,
$$
\sup_{|y|<\sigma}\|f(\cdot+iy)\|_{L^2(\R^n)}\sim  \sup_{|y|<\sigma}\|e^{-y\cdot \xi}\widehat{f}(\xi)\|_{L_\xi^2(\R^n)}=\sup_{|y|<\sigma}\|e^{y\cdot \xi}\widehat{f}(\xi)\|_{L_\xi^2(\R^n)}.
$$
This, recalling the definition of $G^\sigma$ norm, proves \eqref{equ-82-2}.

Now we prove \eqref{equ-82-1}.
We note that \eqref{equ-82-2} implies $\|f\|_{G^\sigma}\lesssim \|e^{\sigma|\xi|}\widehat{f}(\xi)\|_{L_\xi^2(\R^n)}$ if we use the simple fact that  $|y\cdot \xi|\leq |y||\xi|\leq \sigma|\xi|$. Thus it remains to show $\|e^{\frac{\sigma}{2}|\xi|}\widehat{f}(\xi)\|_{L_\xi^2(\R^n)}\lesssim\|f\|_{G^\sigma}$. This, using \eqref{equ-82-2} again, reduces to proving $\|e^{\frac{\sigma}{2}|\xi|}\widehat{f}(\xi)\|_{L_\xi^2(\R^n)}\lesssim \sup_{|y|<\sigma}\|e^{y\cdot \xi}\widehat{f}(\xi)\|_{L_\xi^2(\R^n)}$.  By a scaling argument, it suffices to consider the case $\sigma=1$, namely
\begin{align}\label{equ-82-3}
\|e^{\frac{1}{2}|\xi|}\widehat{f}(\xi)\|_{L_\xi^2(\R^n)}\lesssim \sup_{|y|<1}\|e^{y\cdot \xi}\widehat{f}(\xi)\|_{L_\xi^2(\R^n)}.
\end{align}
In the case $n=1$, this holds clearly, see \cite[p.5285]{Wang19}. But the higher dimension cases need more analysis.
In fact, for every $\xi\in \R^n$,  the Lebesgue measure
$$
\left|\left\{|y|<1: y\cdot \frac{\xi}{|\xi|}\geq \frac{1}{2}\right\}\right|\sim 1.
$$
This implies that for all $\xi\in\R^n$
\begin{align}\label{equ-82-4}
e^{|\xi|}|\widehat{f}(\xi)|^2\lesssim\int_{|y|<1}e^{2y\cdot\xi}|\widehat{f}(\xi)|^2\d y.
\end{align}
Integrating \eqref{equ-82-4} over $\xi\in \R^n$, and using the Fubini theorem, we infer that
$$
\|e^{\frac{1}{2}|\xi|}\widehat{f}(\xi)\|^2_{L_\xi^2(\R^n)}\lesssim\int_{|y|<1}\int_{\R^n}e^{2y\cdot\xi}|\widehat{f}(\xi)|^2\d \xi\d y\lesssim \sup_{|y|<1}\int_{\R^n}e^{2y\cdot\xi}|\widehat{f}(\xi)|^2\d \xi.
$$
This proves \eqref{equ-82-3}, and completes the proof.
\end{proof}

For every $\sigma>0$, we introduce the following analytic function space $A^\sigma$ endowed with the norm
\begin{align}\label{equ-def-A-norm}
\|f\|_{A^\sigma}=\sum_{\alpha\in \mathbb{N}^n}\frac{\sigma^{|\alpha|}\|\partial^\alpha_x f\|_{L^\infty(\R^n)}}{\alpha !}.
\end{align}
\begin{remark}\label{rem-82-1}
If $a$ satisfies $\bf(A1)$, then for all $t\geq 0$, $a(t,\cdot)\in A^{\frac{R}{2}}$. In fact,
$$
 \|a\|_{A^{\frac{R}{2}}}\leq \sum_{\alpha\in \mathbb{N}^n}\frac{(\frac{R}{2})^{|\alpha|}\|\partial^\alpha_x a\|_{L^\infty(\R^n)}}{\alpha !}\leq C\sum_{\alpha\in \mathbb{N}^n}2^{-|\alpha|}\lesssim 1.
$$
\end{remark}

\begin{lemma}\label{lem-pre-2}
For   all $\sigma>0$ and all $a\in A^\sigma,u\in G^\sigma$
$$
\|au\|_{G^{\sigma}}\lesssim \|a\|_{A^\sigma}\|u\|_{G^{\sigma}}.
$$
\end{lemma}
\begin{proof}
Assume that $a\in A^\sigma$. By the Taylor expansion for multi-variable function
$$
a(x+iy)=\sum_{\alpha\in \mathbb{N}^n}\frac{\partial_x^\alpha a(x)}{\alpha!}(iy)^\alpha,
$$
we deduce that
\begin{align}\label{equ-82-5}
\sup_{x\in \R^n, |y|<\sigma}|a(x+iy)|\leq \sum_{\alpha\in \mathbb{N}^n}\left|\frac{\partial_x^\alpha a(x)}{\alpha!}(iy)^\alpha\right| \leq \|a\|_{A^\sigma}.
\end{align}
Recalling the definition of $G^\sigma$ norm, and using \eqref{equ-82-5}, we obtain
$$
\|au\|_{G^{\sigma}}\lesssim \sup_{|y|<\sigma}\|(au)(x+iy)\|_{L_x^2(\R^n)}\leq \|a\|_{A^\sigma}\sup_{|y|<\sigma}\|u(x+iy)\|_{L_x^2(\R^n)}=\|a\|_{A^\sigma}\|u\|_{G^{\sigma}}.
$$
This completes the proof.
\end{proof}

Let $\{e^{-t\Lambda^s}\}_{t\geq 0}$ be the analytic semigroup generated by $-\Lambda^s$ in $L^2(\R^n)$. This semigroup can be expressed by the Fourier transform as
$$
\widehat{e^{-t\Lambda^s}f}(\xi)=e^{-t|\xi|^s}\widehat{f}(\xi), \quad f\in L^2(\R^n).
$$
\begin{lemma}\label{lem-pre-3}
Assume that $s>1$. Then for all $t\geq0$ and all $f\in L^2(\R^n)$
$$
\|e^{-t\Lambda^s}f\|_{G^{t^{\frac{1}{s}}}}\lesssim \|f\|_{L^2(\R^n)}.
$$
\end{lemma}
\begin{proof}
Fix $t>0$. By \eqref{equ-82-1} and the Plancherel theorem, we have
$$
\|e^{-t\Lambda^s}f\|_{G^{t^{\frac{1}{s}}}}\lesssim \|e^{t^{\frac{1}{s}}|\xi|-t|\xi|^s}\widehat{f}(\xi)\|_{L^2(\R^n)}
\leq \|e^{t^{\frac{1}{s}}|\xi|-t|\xi|^s}\|_{L_\xi^\infty(\R^n)}\|f\|_{L^2(\R^n)}.
$$
The desired bound follows from the fact that
$$
 \|e^{t^{\frac{1}{s}}|\xi|-t|\xi|^s}\|_{L_\xi^\infty(\R^n)}=\sup_{s\geq 0} e^{s-s^s}\lesssim 1,
$$
where we used $s>1$ in the last inequality.
\end{proof}

Using the semigroup $\{e^{-t\Lambda^s}\}_{t\geq 0}$, we can rewrite the fractional heat equation \eqref{frac-heat} as an integral equation
\begin{align}\label{frac-heat-1}
u(t)=e^{-t\Lambda^s}u_0+\int_0^te^{-(t-s)\Lambda^s}(au)(s)\d s.
\end{align}
\begin{proposition}
Assume $\bf(A1)$ holds for some $R>0$.  Then there exists a unique solution $u$ of \eqref{frac-heat-1} satisfying
$$
\|u(t,\cdot)\|_{G^{t^{\frac{1}{s}}}}\leq C\exp\left\{Ct\|a\|_{L^\infty(0,\infty;A^{\frac{R}{2}})}\right\}\|u_0\|_{L^2(\R^n)}, \quad \forall t\in [0, (\frac{R}{2})^{s}],
$$
for some constant $C>0$.
\end{proposition}
\begin{proof}
Thanks to Lemma \ref{lem-pre-2} and Lemma \ref{lem-pre-3},
\begin{align}
&\|e^{-t\Lambda^s}u_0\|_{G^{t^{\frac{1}{s}}}}\leq C_0 \|u_0\|_{L^2(\R^n)}, \label{equ-pre-10}\\
&\|au\|_{G^{\sigma}}\leq C_1\|a\|_{A^{\frac{R}{2}}}\|u\|_{G^{\sigma}}, \quad \sigma\in [0,\frac{R}{2}]\label{equ-pre-11}
\end{align}
for some $C_0,C_1>0$. Here $\|a\|_{A^{\frac{R}{2}}}$ is finite, see Remark \ref{rem-82-1}.

Fix $T\in[0,(\frac{R}{2})^{s}]$. Define a ball
$$
\mathcal {B}=\{u: \|u\|_X\leq M\|u_0\|_{L^2(\R^n)}\}
$$
where $M=e^{C_2T}$, $C_2=C_0C_1\sup_{t\geq0}\|a\|_{A^{\frac{R}{2}}}$, and
$$
\|u\|_X=C_0^{-1}\sup_{t\in[0,T]}e^{-C_2t}\|u(t,\cdot)\|_{G^{t^{\frac{1}{s}}}}.
$$

Now we consider the mapping
$$
\Gamma u = e^{-t\Lambda^s}u_0+\int_0^te^{-(t-s)\Lambda^s}(au)(\tau)\d \tau.
$$
If $u\in \mathcal {B}$, then by \eqref{equ-pre-10} and \eqref{equ-pre-11},  we have
\begin{align*}
\|\Gamma u\|_X &\leq  C_0^{-1}\sup_{t\in[0,T]}e^{-C_2t}\|e^{-t\Lambda^s}u_0\|_{G^{t^{\frac{1}{s}}}}\\
& \quad + C_0^{-1}\sup_{t\in[0,T]}e^{-C_2t}\int_0^t \|e^{-(t-s)\Lambda^s}(au)(\tau)\|_{G^{t^{\frac{1}{s}}}}\d \tau\\
&\leq  \|u_0\|+  \sup_{t\in[0,T]}e^{-C_2t}\int_0^t  C_1 \|a\|_{L^\infty(0,\infty;A^{\frac{R}{2}})}\|u(\tau)\|_{G^{\tau^{\frac{1}{s}} }}\d \tau\\
&\leq  \|u_0\|_{L^2(\R^n)}+ \sup_{t\in[0,T]}e^{-C_2t}\int_0^t C_2e^{C_2\tau} \d \tau \|u\|_X\\
& = \|u_0\|_{L^2(\R^n)}+(1-e^{-C_2T})\|u\|_X\\
&\leq \|u_0\|_{L^2(\R^n)}+(1-e^{-C_2T})M\|u_0\|_{L^2(\R^n)}
= M\|u_0\|_{L^2(\R^n)}.
\end{align*}
This means that $\Gamma \mathcal {B}\subset \mathcal {B}$.

Note that in the second inequality above, we have used \eqref{equ-pre-10}-\eqref{equ-pre-11} and the inequality $t^{\frac{1}{s}}-(t-\tau)^{\frac{1}{s}}\leq \tau^{\frac{1}{s}}$ to obtain that
\begin{align*}
 \|e^{-(t-\tau)\Lambda^s}(au)\|_{G^{t^{\frac{1}{s}}}}
\leq C_0\|au\|_{G^{t^{\frac{1}{s}}-(t-\tau)^{\frac{1}{s}}}}
 \leq C_0\|au\|_{G^{\tau^{\frac{1}{s}}}}\leq C_0C_1\sup_{t\geq 0}\|a\|_{A^{\frac{R}{2}}}\|u\|_{G^{\tau^{\frac{1}{s}} }}.
\end{align*}
In the third inequality above, we have used the definitions of $C_2$ and $X$.

Moreover, if $u,v\in \mathcal {B}$, then
$$
\|\Gamma u - \Gamma v\|_X\leq (1-e^{-C_2T})\|u-v\|_X.
$$
Hence,  $\Gamma:\mathcal {B}\mapsto \mathcal {B}$ is a contraction mapping, and   \eqref{frac-heat-1} has a unique solution $u$ in $\mathcal {B}$, namely satisfying  the bound
$$
C_0^{-1}\sup_{t\in[0,T]}e^{-C_2t}\|u(t,\cdot)\|_{G^{t^{\frac{1}{s}}}}\leq M\|u_0\|_{L^2(\R^n)},
$$
which implies that
$$
\|u(T,\cdot)\|_{G^{T^{\frac{1}{s}}}}\leq C_0e^{2C_2T}\|u_0\|_{L^2(\R^n)}.
$$
This gives the desired bound, since $T\in [0,(\frac{R}{2})^{s}]$ is arbitrary.
\end{proof}

\begin{proposition}\label{thm-loc-ana-t}
Assume $\bf(A1)$ holds for some $R>0$. Let $b\in[0,\frac{R}{2}]$. Then there exists a constant $C>0$ so that the solution of \eqref{frac-heat-1} satisfies
$$
\|u(t,\cdot)\|_{G^{t^{\frac{1}{s}}+b}}\leq C\exp\left\{Ct\|a\|_{L^\infty(0,\infty;A^{\frac{R}{2}})}\right\}\|u_0\|_{G^b}, \quad \forall t\in [0,(\frac{R}{2}-b)^s].
$$
\end{proposition}
\begin{proof}
The proof is the same as above. In fact, it suffices to use
$$
\|e^{-t\Lambda^s}u_0\|_{G^{t^{\frac{1}{s}}+b}}\leq C_0 \|u_0\|_{G^b}
$$
instead of \eqref{equ-pre-10}.
\end{proof}

\subsection{Proof of Theorem \ref{thm-ana} (i)}

Now we shall use Proposition \ref{thm-loc-ana-t} repeatedly and an iteration argument to obtain an analytic bound with a fixed analytic radius, which is independent of the time variable.

\begin{lemma}\label{lem-526-1}
Assume that $\bf(A1)$ holds for some $R>0$, then there exist  constants $c,C>0$ so that the solution of \eqref{frac-heat-1} satisfies
$$
\|u(t,\cdot)\|_{G^{cR}}
\leq \exp\left\{C(t^{-1}R^s)^{\frac{1}{s-1}}+Ct\|a\|_{L^\infty(0,\infty;A^{\frac{R}{2}})} \right\}\|u_0\|_{L^2(\R^n)}, \quad t\in(0,(\frac{R}{2})^{s}].
$$
\end{lemma}
\begin{proof}
Let $t\in(0,T]$ with $T=(\frac{R}{2})^{s}$. For every $m\geq 1$, make decomposition
$$
[0,t]=\left[0,\frac{1}{m}t\right]\bigcup\left[\frac{1}{m}t,\frac{2}{m}t\right]\bigcup\cdots \bigcup \left[\frac{(m-1)}{m}t,t\right].
$$
Using Proposition \ref{thm-loc-ana-t} on the intervals $[\frac{j-1}{m}t,\frac{j}{m}t], j=1,2,\cdots,m$, we find
\begin{align*}
\|u(\frac{1}{m}t,\cdot)\|_{G^{(\frac{1}{m}t)^{\frac{1}{s}}}}
&\leq C\exp\left\{C\frac{1}{m}t\|a\|_{L^\infty(0,\infty;A^{\frac{R}{2}})}\right\}\|u_0\|_{L^2(\R^n)}\\
\|u(\frac{2}{m}t,\cdot)\|_{G^{2(\frac{1}{m}t)^{\frac{1}{s}}}}
&\leq C\exp\left\{C\frac{1}{m}t\|a\|_{L^\infty(0,\infty;A^{\frac{R}{2}})}\right\}\|u(\frac{1}{m}t)\|_{G^{(\frac{1}{m}t)^{\frac{1}{s}}}}\\
&\cdots\\
\|u(\frac{j}{m}t,\cdot)\|_{G^{j(\frac{1}{m}t)^{\frac{1}{s}}}}
&\leq C\exp\left\{C\frac{1}{m}t\|a\|_{L^\infty(0,\infty;A^{\frac{R}{2}})}\right\}\|u(\frac{j-1}{m}t)\|_{G^{(j-1)(\frac{1}{m}t)^{\frac{1}{s}}}}\\
&\cdots\\
\|u(t,\cdot)\|_{G^{m(\frac{1}{m}t)^{\frac{1}{s}}}}
&\leq C\exp\left\{C\frac{1}{m}t\|a\|_{L^\infty(0,\infty;A^{\frac{R}{2}})}\right\}\|u(\frac{m-1}{m}t)\|_{G^{(m-1)(\frac{1}{m}t)^{\frac{1}{s}}}}.
\end{align*}
Combining these inequalities we infer that
\begin{align}\label{equ-pre-15}
\|u(t,\cdot)\|_{G^{m(\frac{1}{m}t)^{\frac{1}{s}}}}
\leq C^m\exp\{Ct\|a\|_{L^\infty(0,\infty;A^{\frac{R}{2}})}\}\|u_0\|_{L^2(\R^n)},
\end{align}
provided that
$$
m(\frac{1}{m}t)^{\frac{1}{s}}\leq \frac{R}{2}.
$$
Choose $m\in \mathbb{N}$ so that
$$
m(\frac{1}{m}t)^{\frac{1}{s}} \sim R.
$$
Which is in fact  equivalent to
$$
m\sim  \left(t^{-1}R^s\right)^{\frac{1}{s-1}}.
$$
Then it follows from \eqref{equ-pre-15} that for some $c,C'>0$
$$
\|u(t,\cdot)\|_{G^{cR}}
\leq C^{C'\left(t^{-1}R^s\right)^{\frac{1}{s-1}}}\exp\left\{Ct\|a\|_{L^\infty(0,\infty;A^{\frac{R}{2}})}\right\}\|u_0\|_{L^2(\R^n)}, \quad t\in(0,(\frac{R}{2})^{s}].
$$
This implies the desired bound and completes the proof.
\end{proof}

\begin{proof}[\textbf{Proof of Theorem \ref{thm-ana} (i)}.]
Let $t_0=(\frac{R}{2})^{s}$. It follows from Lemma \ref{lem-526-1} that
\begin{align*}
\|u(t_0,\cdot)\|_{G^{cR}}
&\leq \exp\left\{C(t_0^{-1}R^s)^{\frac{1}{s-1}}+Ct_0\|a\|_{L^\infty(0,\infty;A^{\frac{R}{2}})}\right\}\|u_0\|_{L^2(\R^n)}\\
&\leq \exp\left\{C(1+t_0\|a\|_{L^\infty(0,\infty;A^{\frac{R}{2}})})\right\}\|u_0\|_{L^2(\R^n)}.
\end{align*}
Similarly, we have for all $\tau\geq 0$ that
\begin{align}\label{equ-526-1}
\|u(\tau+t_0,\cdot)\|_{G^{cR}}
\leq \exp\left\{C(1+t_0\|a\|_{L^\infty(0,\infty;A^{\frac{R}{2}})})\right\}\|u(\tau)\|_{L_x^2(\R^n)}.
\end{align}
By the classical energy estimate we have
$$
\|u(\tau)\|_{L^2_x(\R^n)}\leq \exp\{C\tau\|a\|_{L^\infty(0,\infty;A^{\frac{R}{2}})}\}\|u_0\|_{L^2(\R^n)}.
$$
Then we deduce from \eqref{equ-526-1} that for all $t\geq t_0$
\begin{align}\label{equ-526-2}
\|u(t,\cdot)\|_{G^{cR}}
\leq \exp\left\{C(1+t\|a\|_{L^\infty(0,\infty;A^{\frac{R}{2}})})\right\}\|u_0\|_{L^2(\R^n)}.
\end{align}
Moreover, by Lemma \ref{lem-526-1} again, we have for all $t\in(0,t_0]$
\begin{align}\label{equ-526-3}
\|u(t,\cdot)\|_{G^{cR}}
\leq \exp\left\{C\Big[(t^{-1}R^s)^{\frac{1}{s-1}}+t\|a\|_{L^\infty(0,\infty;A^{\frac{R}{2}})}\Big]\right\}\|u_0\|_{L^2(\R^n)}.
\end{align}
Combining \eqref{equ-526-2}-\eqref{equ-526-3}, we infer that for any $t>0$
$$
\|u(t,\cdot)\|_{G^{cR}}
\leq \exp\left\{C\Big[1+(t^{-1}R^s)^{\frac{1}{s-1}}+t\|a\|_{L^\infty(0,\infty;A^{\frac{R}{2}})}\Big]\right\}\|u_0\|_{L^2(\R^n)}.
$$
This implies the bound \eqref{ana-b-1}, since $\|e^{\frac{cR}{2}|\xi|}\widehat{u}(t,\xi)\|_{L^2_\xi(\R^n)}\lesssim\|u(t,\cdot)\|_{G^{cR}}$, which follows from Lemma \ref{lem-pre-1}.
\end{proof}

\subsection{Proof of Theorem \ref{thm-ana} (ii)}
\begin{lemma}\label{lem-526-2}
Assume that $\bf(A2)$ holds. Then there exists $C=C(n,M,\kappa)>0$ so that
$$
\|a\|_{A^R}\leq  e^{CR^{\frac{1}{1-\kappa}}}, \quad \forall R>0.
$$
\end{lemma}
\begin{proof}
By the assumption $\bf(A2)$, we have
  $$
  \sup_{x\in \R^n}|\partial_x^{\alpha}a|\leq CM^{|\alpha|}(\alpha!)^{\kappa}, \quad \forall \alpha\in \mathbb{N}^n.
  $$
By the definition of $A^R$ norm, we get
\begin{align*}
\|a\|_{A^R}&=\sum_{\alpha\in \mathbb{N}^n}\frac{R^{|\alpha|}\|\partial_x^\alpha a\|_{L^\infty(\R^n)}}{\alpha!}\leq \sum_{\alpha\in \mathbb{N}^n} C\frac{(MR)^{|\alpha|}}{(\alpha!)^{1-\kappa}}\\
& = \sum_{\alpha\in \mathbb{N}^n} 2^{-|\alpha|} C\frac{(2MR)^{|\alpha|}}{(\alpha!)^{1-\kappa}} \lesssim  \sup_{\alpha\in \mathbb{N}}\frac{(2MR)^{|\alpha|}}{(\alpha!)^{1-\kappa}}
\\ &\leq   \left(\sup_{\alpha\in \mathbb{N}} \frac{(2MR)^{\alpha}}{(\alpha!)^{1-\kappa}} \right)^n
 = \left(\sup_{\alpha\in \mathbb{N}} \frac{(2MR)^{\frac{\alpha}{1-\kappa}}}{\alpha!}  \right)^{n(1-\kappa)}.
\end{align*}
By the inequality $x^n(n!)^{-1}\leq e^{x}$ for all $x>0,n\in \mathbb{N}$, we find
$$
\sup_{\alpha\in \mathbb{N}} \frac{(2MR)^{\frac{\alpha}{1-\kappa}}}{\alpha!} \leq e^{(2MR)^{\frac{1}{1-\kappa}}}.
$$
Then we conclude that
$$
\|a\|_{A^R}\lesssim e^{n(1-\kappa)(2MR)^{\frac{1}{1-\kappa}}}.
$$
This gives the desired bound.
\end{proof}

\begin{proof}[\textbf{Proof of Theorem \ref{thm-ana} (ii)}.]
According to Lemma \ref{lem-526-2}, we have for some  $C>0$
\begin{align}\label{equ-83-1}
\sup_{t>0}\|a\|_{A^R}\leq Ce^{CR^{\frac{1}{1-\kappa}}}, \quad \forall R>0.
\end{align}
Let $u$ be the solution of \eqref{frac-heat}. Fix $t>0$. By Theorem \ref{thm-ana} (i), there exists $C>0$ so that
\begin{align*}
\left\|\widehat{u}(t,\cdot)e^{C^{-1}R|\xi|}\right\|_{L_\xi^2(\R^n)} &\leq \exp\left\{C\Big[1+\left(t^{-1}R^s\right)^{\frac{1}{s-1}}+t\sup_{t>0}\|a(t,\cdot)\|_{A^{\frac{R}{2}}}\Big]\right\} \|u_0\|_{L^2(\R^n)}\\
&\leq \exp\left\{C(t^{-\frac{1}{s-1}}+t)(R^{\frac{s}{s-1}}+\sup_{t>0}\|a(t,\cdot)\|_{A^{\frac{R}{2}}})\right\} \|u_0\|_{L^2(\R^n)}
\end{align*}
holds for all  $R\geq 1$.
By using  \eqref{equ-83-1} and absorbing the term $R^\frac{s}{s-1}$, we have
\begin{align}\label{equ-527-1}
\int_{\R^n}e^{2C^{-1}R|\xi|}|\widehat{u}(t,\xi)|^2\d \xi \leq e^{C(t^{-\frac{1}{s-1}}+t)}e^{Ce^{CR^{\frac{1}{1-\kappa}}}}\|u_0\|^2_{L^2(\R^n)}, \quad \forall R\geq 1.
\end{align}
Rewrite \eqref{equ-527-1} as
\begin{align}\label{equ-527-2}
e^{-2Ce^{CR^{\frac{1}{1-\kappa}}}}\int_{\R^n}e^{2C^{-1}R|\xi|}|\widehat{u}(t,\xi)|^2\d \xi \leq e^{C(t^{-\frac{1}{s-1}}+t)}e^{-Ce^{CR^{\frac{1}{1-\kappa}}}}\|u_0\|^2_{L^2(\R^n)}, \quad \forall R\geq 1.
\end{align}
Integrating \eqref{equ-527-2} over $R\in[1,\infty)$ and changing the integration order, we obtain
\begin{align}\label{equ-527-3}
\int_{\R^n}\int_{1}^\infty e^{-2Ce^{CR^{\frac{1}{1-\kappa}}}}e^{2C^{-1}R|\xi|}\d R|\widehat{u}(t,\xi)|^2\d \xi \leq C e^{C(t^{-\frac{1}{s-1}}+t)}\|u_0\|^2_{L^2(\R^n)}.
\end{align}
Thanks to \eqref{equ-527-3}, Theorem \ref{thm-ana} (ii) holds true if one can show the following  claim: There exists  $c>0$ so that
\begin{align}\label{equ-527-4}
 \int_{1}^\infty e^{-2Ce^{CR^{\frac{1}{1-\kappa}}}}e^{2C^{-1}R|\xi|}\d R \geq ce^{c|\xi|(\log(e+|\xi|))^{1-\kappa}}, \quad \forall \xi\in\R^n.
\end{align}

Finally, it remains to show \eqref{equ-527-4}. Without loss of generality, we assume $C>1$.
In fact, in the case $|\xi|\leq 4C^2e^{2^{\frac{1}{1-\kappa}}C}$, we have
\begin{align}\label{equ-527-5}
\int_{1}^\infty e^{-2Ce^{CR^{\frac{1}{1-\kappa}}}}e^{2C^{-1}R|\xi|}\d R\geq \int_{1}^2 e^{-2Ce^{CR^{\frac{1}{1-\kappa}}}}e^{2C^{-1}R|\xi|}\d R \gtrsim e^{|\xi|\log(e+|\xi|)}.
\end{align}
This proves \eqref{equ-527-4}.

In the case $|\xi|>4C^2e^{2^{\frac{1}{1-\kappa}}C}$, let
\begin{align}\label{equ-527-6}
R_0=\left(C^{-1}\log \frac{|\xi|}{4C^2} \right)^{1-\kappa}>2.
\end{align}
If $1\leq R\leq R_0$, one can check that
$$
C^{-1}R|\xi|>C^{-1}R\cdot4C^2e^{2^{\frac{1}{1-\kappa}}C}> 2Ce^{CR^{\frac{1}{1-\kappa}}},
$$
which is equivalent to
\begin{align}\label{equ-527-7}
e^{-2Ce^{CR^{\frac{1}{1-\kappa}}}}e^{2C^{-1}R|\xi|}\geq e^{C^{-1}R|\xi|}.
\end{align}
It follows from \eqref{equ-527-6}-\eqref{equ-527-7} that \eqref{equ-527-7} holds if $R\in[R_0-1,R_0]$. Thus,
\begin{align}\label{equ-527-8}
\int_{1}^\infty e^{-2Ce^{CR^{\frac{1}{1-\kappa}}}}e^{2C^{-1}R|\xi|}\d R\geq \int_{R_0-1}^{R_0} e^{C^{-1}R|\xi|}\d R \geq e^{C^{-1}(R_0-1)|\xi|}.
\end{align}
By \eqref{equ-527-6} again, we have
$$
R_0 \sim \left(\log(e+|\xi|)\right)^{1-\kappa}, \quad |\xi|\to \infty.
$$
This implies that for some small $c>0$
$$
e^{C^{-1}(R_0-1)|\xi|}\geq ce^{c|\xi|\left(\log(e+|\xi|)\right)^{1-\kappa}}, \quad |\xi|>4C^2e^{2^{\frac{1}{1-\kappa}}C}.
$$
This, together with \eqref{equ-527-8}, shows that \eqref{equ-527-4} holds.

\end{proof}

\section{Quantitative unique continuation for analytic functions}

In this section, we first prove some interpolation inequalities on thick sets for analytic functions and ultra-analytic functions, respectively; and  then we prove Theorem \ref{thm-interpolation}. To this end, we first recall a local interpolation inequality for analytic functions.
\begin{lemma}[{\cite[Theorem 1.3]{AE} }]\label{lem-AE}
Let $R>0$ and let  $f: B_{2R}\subset \mathbb{R}^n\rightarrow \mathbb{R}$ be real analytic in $B_{2R}$ verifying
$$
\left|\partial_x^\alpha f(x)\right|\leq M(\rho R)^{-|\alpha|}|\alpha|!,\;\;\mbox{when}\;\;  x\in B_{2R}\;\;
\mbox{and}\;\; \alpha\in\mathbb{N}^n
$$
for some positive numbers $M$ and $\rho\in (0,1]$.
Let $\omega\subset B_R$ be a subset of positive measure. Then there are  constants  $C=C(\rho,|\omega|/|B_R|)>0$ and $\theta=\theta(\rho,|\omega|/|B_R|)\in (0,1)$
so that
\begin{equation}\label{equ-522-1}
\|f\|_{L^\infty(B_R)}\leq CM^{1-\theta}\Big(\frac{1}{|\omega|}\int_{\omega}|f(x)|\, \d x\Big)^{\theta}.
\end{equation}
\end{lemma}

Here and in the sequel, we use $B_R$ to denote a ball in $\R^n$ with radius $R$, $Q_L$ a cube in $\R^n$ with side length $L$.

\begin{remark}\label{rem-AE}
As a consequence of \eqref{equ-522-1}, we can derive an $L^p$ version inequality, which will be useful later.
Let $1\leq p<\infty$ and $\frac{1}{p}+\frac{1}{p'}=1$. By the H\"{o}lder inequality, we have
$$
|B_R|^{\frac{1}{p}}\|f\|_{L^\infty(B_R)}\geq \|f\|_{L^p(B_R)}, \quad \frac{1}{|\omega|}\int_{\omega}|f(x)|\, \d x \leq |\omega|^{\frac{1}{p'}-1}\|f\|_{L^p(\omega)}.
$$
Inserting them into \eqref{equ-522-1} we obtain
\begin{align}\label{equ-522-2}
\|f\|_{L^p(B_R)}\leq C|B_R|^{\frac{1}{p}}|\omega|^{\theta(\frac{1}{p'}-1)}M^{1-\theta}\|f\|^\theta_{L^p(\omega)}.
\end{align}
Clearly, \eqref{equ-522-2} also holds in the case that $p=\infty$.

\end{remark}


Based on Lemma \ref{lem-AE}, we establish an interpolation inequality of unique continuation  for functions in $G^\sigma$, with an explicit  index $\theta$ depending on $\sigma$.

\begin{lemma}\label{lem-prop-small-local}
Let $2\leq p\leq \infty, L,\sigma>0$ and $\omega\subset Q_L$ be a subset of positive measure. Then there exist two constants $C=C(p,n,|\omega|,L,\sigma)>0$ and $C'=C'(n,|\omega|,L)>0$ so that
$$
\|f\|_{L^p(Q_L)}\leq C\|f\|_{L^p(\omega)}^\theta M^{1-\theta}
$$
holds for all  $\theta\in(0,e^{-C'\max\{1,\frac{L}{\sigma}\}})$, where
$$
M=\sup_{\alpha\in \mathbb{N}^n}\frac{\sigma^{|\alpha|}}{|\alpha|!}\|\partial_x^\alpha f\|_{L^\infty(Q_{2L})}.
$$
\end{lemma}
\begin{proof}
Let $f\in G^\sigma$ with some $\sigma>0$. Clearly, $f$  is real analytic on $\R^n$. Also, by the definition of $M$, we have
\begin{align}\label{equ-523-4}
\|\partial_x^\alpha f\|_{L^\infty(Q_{2L})}\leq M\left(\frac{1}{\sigma}\right)^{|\alpha|}|\alpha|!, \quad \mbox{ for all } \alpha\in \mathbb{N}^n.
\end{align}

If $\sigma\geq L$, then \eqref{equ-523-4} holds with $\left(\frac{1}{\sigma}\right)^{|\alpha|}$ replaced by $\left(\frac{1}{L}\right)^{|\alpha|}$. Since $\omega\subset Q_L$ has a positive measure, we can apply Lemma \ref{lem-AE} and Remark \ref{rem-AE} (with $\rho=1$, $R=L$) to obtain that
$$
\|f\|_{L^p(Q_L)}\leq C\|f\|_{L^p(\omega)}^{\theta_0} M^{1-\theta_0}
$$
for some $\theta_0\in(0,1)$. Thus this lemma holds in this case.

Now we consider the other  case that  $\sigma<L$. We first claim that there exists  a point $x_0\in Q_L$ so that $Q_{\frac{\sigma}{3}}(x_0)\subset Q_L$ and
\begin{align}\label{equ-724-1}
\frac{|\omega\bigcap Q_{\frac{\sigma}{3}}(x_0)|}{|Q_{\frac{\sigma}{3}}(x_0)|}\geq c_0
\end{align}
for some $c_0=c_0(n,L,|\omega|)>0$. In fact, let $k\geq 1$ be an integer, we split $Q_L$ as disjoint small cubes $Q_{\frac{L}{k}}(x)$, then
$$
|\omega|=\sum_{x}|\omega \bigcup Q_{\frac{L}{k}}(x)|.
$$
Choose an $x_0$ so that $|\omega \bigcup Q_{\frac{L}{k}}(x_0)|=\max_x |\omega \bigcup Q_{\frac{L}{k}}(x)|$. Since there are $k^n$ small cubes, we infer that
\begin{align}\label{equ-724-2}
|\omega \bigcup Q_{\frac{L}{m}}(x_0)|\geq k^{-n}|\omega|.
\end{align}
Set $k=[\frac{3L}{\sigma}]+1$. Then $\frac{L}{k}\leq \frac{\sigma}{3}$, and of course $Q_{\frac{L}{k}}(x_0)\subset Q_{\frac{\sigma}{3}}$. Thus \eqref{equ-724-2} becomes
$$
|\omega \bigcup Q_{\frac{\sigma}{3}}(x_0)|\geq k^{-n}|\omega|,
$$
which, noting $k\leq \frac{4L}{\sigma}$, implies that
\begin{align*}
\frac{|\omega\bigcap Q_{\frac{\sigma}{3}}(x_0)|}{|Q_{\frac{\sigma}{3}}(x_0)|}\geq \frac{k^{-n}|\omega|}{(\frac{\sigma}{3})^n}
\geq \left(\frac{3}{4L}\right)^n |\omega|.
\end{align*}
This proves the claim \eqref{equ-724-1}.

Since the proof in the case $p=\infty$ is similar, we can now assume that  $2\leq p<\infty$. From \eqref{equ-724-1} and the bound \eqref{equ-523-4}, we apply Lemma \ref{lem-AE} and Remark \ref{rem-AE} (with $\rho=1$, $R=\sigma$, $\omega$ replaced by $\omega\bigcap Q_{\frac{\sigma}{3}}$) to obtain that
\begin{align}\label{equ-523-5}
\int_{Q_{\sigma}(x_0)}|f(x)|^p\d x\leq C\left( \int_{\omega\bigcap Q_{\frac{1}{3}\sigma}(x_0)}|f(x)|^p\d x\right)^\delta M^{p(1-\delta)},
\end{align}
and similarly for all $Q_\sigma(y)\subset Q_L$
\begin{align}\label{equ-523-6}
\int_{Q_{\sigma}(y)}|f(x)|^p\d x\leq C\left( \int_{ Q_{\frac{1}{3}\sigma}(y)}|f(x)|^p\d x\right)^\delta M^{p(1-\delta)},
\end{align}
where $\delta=\delta(n,L,|\omega|)\in(0,1)$ and $C=C(p,n,L,|\omega|,\sigma)>0$. We point out that $\delta$ is independent of $\sigma$, which is important in our proof.

On the one hand, it follows from \eqref{equ-523-5} that
\begin{align}\label{equ-523-7}
\int_{Q_{\frac{1}{3}\sigma}(x_0)}|f(x)|^p\d x\leq C\left( \int_{\omega}|f(x)|^p\d x\right)^\delta M^{p(1-\delta)}.
\end{align}

On the other hand, it follows from  \eqref{equ-523-6} that
\begin{align}\label{equ-523-8}
\int_{Q_{\frac{1}{3}\sigma}(y)}|f(x)|^p\d x\leq C\left( \int_{ Q_{\frac{1}{3}\sigma}(y')}|f(x)|^p\d x\right)^\delta M^{p(1-\delta)}
\end{align}
for all $y,y'\in \R^n$ satisfying $|y-y'|\leq \frac{\sigma}{3}$ and $Q_{\frac{1}{3}\sigma}(y), Q_{\frac{1}{3}\sigma}(y')\subset Q_L$. With \eqref{equ-523-8} in hand, for every $m\in \mathbb{N}$, we can use the Harnack chain argument to prove that
\begin{align}\label{equ-523-9}
\int_{Q_{\frac{1}{3}\sigma}(y)}|f(x)|^p\d x\leq C\left( \int_{ Q_{\frac{1}{3}\sigma}(y')}|f(x)|^p\d x\right)^{\delta^m} M^{p(1-\delta^m)}
\end{align}
for all $y,y'\in \R^n$ satisfying $|y-y'|\leq \frac{m\sigma}{3}$ and $Q_{\frac{1}{3}\sigma}(y), Q_{\frac{1}{3}\sigma}(y')\subset Q_L$. Since $x_0\in Q_L$, we have
$$
|y-x_0|\leq \sqrt{n}L, \quad \mbox{ for all } y\in Q_L.
$$
This, together with \eqref{equ-523-9} (setting $y'=x_0$), implies that for all $Q_\sigma(y)\subset Q_L$
\begin{align}\label{equ-523-9.5}
\int_{Q_{\frac{1}{3}\sigma}(y)}|f(x)|^p\d x\leq C\left( \int_{ Q_{\frac{1}{3}\sigma}(x_0)}|f(x)|^p\d x\right)^{\delta^m} M^{p(1-\delta^m)},
\end{align}
where $m=[\frac{3\sqrt{n}L}{\sigma}]+1$. Integrating \eqref{equ-523-9.5} over $\{y\in Q_L: Q_\sigma(y)\subset Q_L\}$, we infer that
\begin{align}\label{equ-523-10}
\int_{Q_{L}}|f(x)|^p\d x\leq C\left( \int_{ Q_{\frac{1}{3}\sigma}(x_0)}|f(x)|^p\d x\right)^{\delta^m} M^{p(1-\delta^m)},
\end{align}
for some constant $C>0$.

Finally, combining \eqref{equ-523-7} and \eqref{equ-523-10}, we get
\begin{align}\label{equ-523-11}
\int_{Q_L}|f(x)|^p\d x\leq C\left( \int_{\omega}|f(x)|^p\d x\right)^{\delta^{m+1}} M^{p(1-\delta^{m+1})}.
\end{align}
Recall that $m\leq \frac{L}{\sigma}(3\sqrt{n}+1)$, we have for some $C'=C(n,L,|\omega|)>0$
$$
 \delta^{m+1}=e^{-(m+1)\log \delta^{-1}}\geq e^{-C'\frac{L}{\sigma}}:=\theta.
$$
This, together with \eqref{equ-523-11} and the trivial bound $\int_{\omega}|f(x)|^p\d x\leq M^p$, gives that
$$
\int_{Q_L}|f(x)|^p\d x\leq C\left( \int_{\omega}|f(x)|^p\d x\right)^{\theta} M^{p(1-\theta)}.
$$
This completes the proof.
\end{proof}

For our purpose, we need to bound the quantity $M$ in Lemma \ref{lem-prop-small-local} in terms of $\|f\|_{G^\sigma}$. To this end,  for every $j\in \mathbb{Z}^n$, we define
\begin{align}\label{equ-723-1}
M_j=\sup_{\alpha\in \mathbb{N}^n}\frac{\sigma^{|\alpha|}}{|\alpha|!}\|\partial_x^\alpha f\|_{L^\infty(Q_{2L}(jL))}.
\end{align}
Here we use the convention notation $jL=(j_1L,j_2L,\cdots, j_nL)\in\R^n$.
\begin{lemma}\label{lem-M}
Let $p\geq 2$ and $\sigma,L>0$. Then there exists $C=C(n,\sigma,L)>0$ so that
$$
\left(\sum_{j\in \mathbb{Z}^n}M_j^p \right)^{\frac{1}{p}}\leq C\|f\|_{G^{4\sigma}},  \quad  \mbox{ for all } f\in G^{4\sigma}.
$$
\end{lemma}
\begin{proof} Thanks to the inequality
$$
\left(\sum_{j\in \mathbb{Z}^n}M_j^p\right)^{\frac{1}{p}}\leq  \left(\sum_{j\in \mathbb{Z}^n}M_j^2\right)^{\frac{1}{2}},
$$
it suffices to consider the case $p=2$. Using the definition \eqref{equ-723-1}, we have
\begin{align}\label{equ-724-11}
\sum_{j\in \mathbb{Z}^n}M_j^2&=\sum_{j\in \mathbb{Z}^n}\left( \sup_{\alpha\in \mathbb{N}^n}\frac{\sigma^{|\alpha|}}{|\alpha|!}\|\partial_x^\alpha f\|_{L^\infty(Q_{2L}(jL))}\right)^2.
\end{align}

We claim that there exists $C=C(n,L)>0$ so that
\begin{align}\label{equ-523-0}
\sup_{\alpha\in \mathbb{N}^n}\frac{\sigma^{|\alpha|}}{|\alpha|!}\|\partial_x^\alpha f\|_{L^\infty(Q_{2L})}\leq C(1+\sigma^{-n})\sup_{\alpha\in \mathbb{N}^n}\frac{(2\sigma)^{|\alpha|}}{|\alpha|!}\|\partial_x^\alpha f\|_{L^2(Q_{2L})}.
\end{align}
In fact, since $Q_{2L}$ satisfies the cone property,  by the Sobolev embeding $\|f\|_{L^\infty(Q_{2L})}\leq C \|f\|_{H^n(Q_{2L})}$, we have
\begin{align}\label{equ-523-1}
\|\partial_x^\alpha f\|_{L^\infty(Q_{2L})} \leq C\sum_{\beta\in \mathbb{N}^n, |\beta|\leq n} \|\partial_x^{\alpha +\beta} f\|_{L^2(Q_{2L})}.
\end{align}
For $\beta\in \mathbb{N}^n, |\beta|\leq n$
\begin{align*}
\|\partial_x^{\alpha +\beta} f\|_{L^2(Q_L)} &\leq (2\sigma)^{-|\alpha+\beta|}(\alpha+\beta)! \sup_{\alpha\in \mathbb{N}^n}\frac{(2\sigma)^{|\alpha|}}{|\alpha|!}\|\partial_x^\alpha f\|_{L^2(Q_{2L})}\\
&\leq C (1+\sigma^{-n}) \sigma^{-|\alpha|}|\alpha|!\sup_{\alpha\in \mathbb{N}^n}\frac{(2\sigma)^{|\alpha|}}{|\alpha|!}\|\partial_x^\alpha f\|_{L^2(Q_{2L})}.
\end{align*}
Here we used the facts that $\sigma^{-|\alpha+\beta|}\leq \sigma^{-|\alpha|}(1+\sigma^{-n})$ and $2^{-|\alpha|}(\alpha+\beta)!\leq C|\alpha|!$. This, together with \eqref{equ-523-1}, gives  the bound \eqref{equ-523-0}.

Note that \eqref{equ-523-0} holds for all $Q_{2L}(jL), j\in \mathbb{Z}^n$, we deduce from \eqref{equ-724-11} that
\begin{align*}
\sum_{j\in \mathbb{Z}^n}M_j^2&\lesssim\sum_{j\in \mathbb{Z}^n}\left(\sup_{\alpha\in \mathbb{N}^n}\frac{(2\sigma)^{|\alpha|}}{|\alpha|!}\|\partial_x^\alpha f\|_{L^2(Q_{2L}(jL))}\right)^2\\
&\lesssim\sup_{\alpha\in \mathbb{N}^n}\left(\frac{(2\sigma)^{|\alpha|}}{|\alpha|!}\right)^2\sum_{j\in \mathbb{Z}^n}\|\partial_x^\alpha f\|^2_{L^2(Q_{2L}(jL))}\\
&\lesssim\left(\sup_{\alpha\in \mathbb{N}^n}\frac{(2\sigma)^{|\alpha|}}{|\alpha|!}\|\partial_x^\alpha f\|_{L^2(\R^n)}\right)^2.
\end{align*}
Here the implicit constant depends only on $n,\sigma,L$. Then the lemma follows if we can show that
\begin{align}\label{equ-83-2}
\sup_{\alpha\in \mathbb{N}^n}\frac{(2\sigma)^{|\alpha|}}{|\alpha|!}\|\partial_x^\alpha f\|^2_{L^2(\R^n)}\lesssim \|u\|_{G^{4\sigma}}.
\end{align}
In fact, by the Plancherel theorem, we have
\begin{align}\label{equ-83-3}
\sup_{\alpha\in \mathbb{N}^n}\frac{(2\sigma)^{|\alpha|}}{|\alpha|!}\|\partial_x^\alpha u\|^2_{L^2(\R^n)}\sim \sup_{\alpha\in \mathbb{N}^n}\frac{(2\sigma)^{|\alpha|}}{|\alpha|!}\|(i\xi)^\alpha \widehat{u}\|^2_{L^2(\R^n)}\lesssim \|e^{2\sigma|\xi|}\widehat{u}(\xi)\|_{L^2(\R^n)},
\end{align}
where we have used $(2\sigma)^{|\alpha|}|(i\xi)^\alpha|\leq (2\sigma|\xi|)^{|\alpha|}\leq |\alpha|!e^{2\sigma|\xi|}$. The bound \eqref{equ-83-3} and Lemma \ref{lem-pre-1} imply \eqref{equ-83-2}. Thus the proof is completed.
\end{proof}


We now present the following H\"{o}lder type inequality of unique continuation on thick sets  for functions in $G^\sigma$.
\begin{theorem}\label{thm-un-ana}
Let $2\leq p\leq \infty,\sigma>0$ and $E$ be a thick set in $\R^n$.  Then there exist two constants $C=C(p,n,E,\sigma)>0$ and $C'=C'(n,E)>0$ so that
\begin{align}\label{equ-722-1}
\|f\|_{L^p(\R^n)}\leq C\|f\|_{L^p(E)}^\theta\|f\|^{1-\theta}_{G^\sigma}
\end{align}
holds for all $f\in G^\sigma$ and all $\theta\in (0,e^{-C'\max\{1,\sigma^{-1}\}})$.
\end{theorem}

In particular, by letting $p=2$, we obtain the following result.
\begin{corollary}\label{cor-un-ana}
Let $\sigma>0$ and $E$ be a thick set in $\R^n$.  Then there exist two constants $C=C(n,E,\sigma)>0$ and $C'=C'(n,E)>0$ so that
\begin{align}\label{equ-725-11}
\int_{\R^n}|f(x)|^2\d x\leq C\left(\int_E|f(x)|^2\d x \right)^\theta\|f\|^{2(1-\theta)}_{G^\sigma}
\end{align}
holds for all $f\in G^\sigma$ and $\theta\in(0,e^{-C'\max\{1,\sigma^{-1}\}})$.
\end{corollary}

Before proving Theorem \ref{thm-un-ana}, we give two remarks below.

\begin{remark}
The inequality \eqref{equ-722-1} fails in the case $1\leq p<2$. Given now $1\leq p<2$. Since $\|f\|_{L^p(E)}\leq \|f\|_{L^p(\R^n)}$, this claim follows if we can disprove
\begin{align}\label{equ-725-10}
\|f\|_{L^p(\R^n)}\leq C\|f\|_{G^\sigma}, \quad \forall f\in G^\sigma.
\end{align}
To this end, for every $s>0$, define a function
$$
f_s(x)=(1+|x|^2)^{-\frac{s}{2}}, \quad x\in \R^n.
$$
Then $f_s\in L^p(\R^n)$ if and only if $s>\frac{n}{p}$. Moreover, by \cite[Proposition 6.1.5, p.6]{Gra}, if $s<n$, then the Fourier transform $\widehat{f_s}$ satisfies that
$$
|\widehat{f_s}(\xi)|\leq
\left\{
\begin{array}{ll}
Ce^{-\frac{1}{2}|\xi|}, & \quad |\xi|\geq 2,\\
C|\xi|^{s-n}, & \quad |\xi|\leq 2.
\end{array}
\right.
$$
Then $f_s\in G^{\frac{1}{4}}$ if $s\in(\frac{n}{2},n)$. Since $p<2$, we can always choose an $s_0$ so that $s_0\leq \frac{n}{p}$ and $s_0\in(\frac{n}{2},n)$. Then $f_{s_0}\in G^{\frac{1}{4}}$ but $f_{s_0}\notin L^p(\R^n)$, this shows that  \eqref{equ-725-10} fails to hold in the case $\sigma=\frac{1}{4}$. We conclude the same result for the general $\sigma>0$ after a scaling argument.
\end{remark}

\begin{remark}\label{rem-ana-inter}
We recall here some previous works on the interpolation inequality \eqref{equ-725-11}.
\begin{itemize}
  \item In \cite{Lebeau}, $E$ is thick, but no explicit dependence of $\theta$ on $\sigma$, proved by Carleman estimates.
  \item In \cite{WWZ}, $E$ is the complement set of a ball, $\theta\sim e^{-1/\sigma}$, proved by three ball inequality of analytic functions.
  \item In \cite{Bourgain,Han}, $E$ is a Borel set satisfying the thick condition, $\theta\sim e^{-1/\sigma}$, proved by harmonic measure estimate.
\end{itemize}
Note that the complement set of every ball is a thick set, every Borel set is a Lebesgue measurable set (but the converse is not ture), Corollary \ref{cor-un-ana} covers the results in \cite{Bourgain,Lebeau,WWZ,Han} in a unified way.
\end{remark}

\begin{proof}[\textbf{Proof of Theorem \ref{thm-un-ana}}.]
Let $E$ be a thick set. Then there exists $L>0$ so that
\begin{align}\label{equ-725-1}
\inf_{j\in \mathbb{Z}^n}|E\bigcap Q_{L}(jL)|=C_0>0.
\end{align}
Let $\sigma>0$.  Since the lower bound $C_0$ is independent of $j$, we apply Lemma \ref{lem-prop-small-local} with $\omega=E\bigcap Q_{L}(jL)$ to find that
\begin{align}\label{equ-725-2}
\|f\|_{L^p(Q_L(jL))}\leq C\|f\|_{L^p(E\bigcap Q_{L}(jL))}^\theta M_j^{1-\theta},
\end{align}
where $\theta=e^{-C'\max\{1,L/\sigma\}}\in(0,1)$, $C=C(p,n,|\omega|,L,\sigma)>0$, $C'=C'(n,|\omega|,L)>0$ and
$$
M_j=\sup_{\alpha\in \mathbb{N}^n}\frac{\sigma^{|\alpha|}}{|\alpha|!}\|\partial_x^\alpha f\|_{L^\infty(Q_{2L}(jL))}.
$$

The proof splits into two cases.

{\bf Case (1): $p=\infty$.} By Lemma \ref{lem-pre-1}, we know
$$
\sup_{j\in \mathbb{Z}^n}M_j\leq C\|f\|_{G^{4\sigma}},
$$
which, together with \eqref{equ-725-2}, shows that
\begin{align*}
\|f\|_{L^\infty(\R^n)}\leq \sup_{j\in \mathbb{Z}^n}\|f\|_{L^\infty(Q_L(jL))}\leq C\sup_{j\in \mathbb{Z}^n}\|f\|_{L^\infty(E\bigcap Q_{L}(jL))}^\theta M_j^{1-\theta}\leq \|f\|_{L^\infty(E)}^\theta\|f\|^{1-\theta}_{G^{4\sigma}}.
\end{align*}
Note that this   holds for all $\sigma>0$, replacing $4\sigma$ by $\sigma$, we conclude \eqref{equ-722-1} in this case.

{\bf Case (2): $2\leq p<\infty$.} Taking the $p$-th power of \eqref{equ-725-2} we obtain
\begin{align}\label{equ-725-3}
\int_{Q_L(jL)}|f(x)|^p\d x\leq C\left( \int_{E\bigcap Q_{L}(jL)}| f(x)|^p\d x\right)^\theta M_j^{p(1-\theta)}, \quad \forall j\in \mathbb{Z}^n.
\end{align}
Using the decomposition
$$
\int_{\R^n}|f(x)|^p\d x= \sum_{j\in \mathbb{Z}^n}\int_{Q_L(jL)}|f(x)|^p\d x
$$
and the bound \eqref{equ-725-3}, we deduce
\begin{align}\label{equ-725-4}
\int_{\R^n}|f(x)|^p\d x&\leq \sum_{j\in \mathbb{Z}^n}C\left( \int_{E\bigcap Q_{L}(jL)}| f(x)|^p\d x\right)^\theta M_j^{p(1-\theta)}\nonumber\\
&\leq C\left(\varepsilon \sum_{j\in \mathbb{Z}^n} M_j^p + \varepsilon^{-\frac{1-\theta}{\theta}} \sum_{j\in \mathbb{Z}^n}   \int_{E\bigcap Q_{L}(jL)}| f(x)|^p\d x\right)
\end{align}
for all $\varepsilon>0$.

On one hand,
\begin{align}\label{equ-725-5}
 \sum_{j\in \mathbb{Z}^n}   \int_{E\bigcap Q_{L}(jL)}| f(x)|^p\d x = \int_{E}|f(x)|^p\d x.
\end{align}

On the other hand, thanks to Lemma \ref{lem-M}, we have
\begin{align}\label{equ-725-6}
\sum_{j\in \mathbb{Z}^n} M_j^p \leq C\|f\|^p_{G^{4\sigma}}.
\end{align}
Inserting \eqref{equ-725-5}-\eqref{equ-725-6} into \eqref{equ-725-4}, we have
\begin{align}\label{equ-725-7}
\int_{\R^n}|f(x)|^p\d x\leq C\left(\varepsilon \|f\|^p_{G^{2\sigma}} + \varepsilon^{-\frac{1-\theta}{\theta}} \int_{E} |f(x)|^p\d x\right)
\end{align}
for all $\varepsilon>0$, $C>0$ is a different constant.

By taking $\varepsilon=\varepsilon_0$ so that
$$
\varepsilon_0 \|f\|^p_{G^{2\sigma}} = \varepsilon_0^{-\frac{1-\theta}{\theta}} \int_{E} |f(x)|^p\d x,
$$
we deduce from \eqref{equ-725-7} that
\begin{align*}
\int_{\R^n}|f(x)|^p\d x\leq C\left(\int_E|f(x)|^p \d x\right)^\theta\|f\|^{p(1-\theta)}_{G^{4\sigma}}.
\end{align*}
This shows that \eqref{equ-722-1} holds for $2\leq p<\infty$. It completes the proof.
\end{proof}


To prove   interpolation  inequalities for ultra-analytic functions, we proceed with a different approach, which relies on the following uncertainty principle.

\begin{theorem}\label{thm-un}
Let $E$ be a thick set in $\R^n$ and let $N>0$. Then
$$
\int_{\R^n}|f(x)|^2\d x \leq Ce^{CN}\int_{E}|f(x)|^2\d x
$$
holds for all $f\in L^2(\R^n), supp \widehat{f}\subset B_N(0).$
\end{theorem}
\begin{proof}
This is the classical Logvinenko-Sereda theorem, see e.g. \cite{HJ,Kov}.
\end{proof}

\begin{theorem}\label{thm-inter-super}
Let $c>0,\delta>0$ and $E$ be a thick set in $\R^n$. Then there exists $C>0$ so that
\begin{align}\label{equ-84-0}
\|f\|_{L^2(\R^n)} \leq Ce^{e^{C\left(\frac{\theta}{1-\theta} \right)^{\frac{1}{\delta}}}} \|f\|^\theta_{L^2(E)}
  \|e^{c|\xi|(\log(e+|\xi|))^\delta}\widehat{f}(\xi)\|^{1-\theta}_{L^2_\xi(\R^n)}
\end{align}
holds  for any $\theta\in(0,1)$ and $f$ satisfying $\|e^{c|\xi|(\log(e+|\xi|))^\delta}\widehat{f}(\xi)\|_{L^2_\xi(\R^n)}<\infty$.
\end{theorem}
\begin{proof}
The proof relies on a high-low frequency decomposition. Let $N>0$. Define $P_{\leq N}$ and $P_{>N}$ as
$$
\widehat{P_{\leq N}f}=\chi_{|\xi|\leq N}\widehat{f}, \quad \widehat{P_{> N}f}=\chi_{|\xi|> N}\widehat{f},
$$
where $\chi_{A}$ denotes the characteristic functions of the set $A$. By Theorem \ref{thm-un},
\begin{align}\label{equ-528-1}
\|f\|_{L^2(\R^n)}&\leq \|P_{\leq N}f\|_{L^2(\R^n)}+\|P_{>N}f\|_{L^2(\R^n)}\nonumber\\
&\leq Ce^{CN}\|P_{\leq N}f\|_{L^2(E)}+\|P_{>N}f\|_{L^2(\R^n)}\nonumber\\
&\leq Ce^{CN}\|f\|_{L^2(E)}+(1+Ce^{CN})\|P_{>N}f\|_{L^2(\R^n)}\nonumber\\
&\leq Ce^{CN}\|f\|_{L^2(E)}+(1+Ce^{CN})e^{-c N(\log(e+N))^\delta}\|e^{c|\xi|(\log(e+|\xi|))^\delta}\widehat{f}(\xi)\|_{L^2_\xi(\R^n)}.
\end{align}

Arbitrarily fix $\varepsilon\in(0,1)$, and then choose $N_0$ so that
\begin{align}\label{equ-528-2}
(1+Ce^{CN_0})e^{-c N_0(\log(e+N_0))^\delta}=\varepsilon.
\end{align}
This is always possible since the set $\{(1+Ce^{CN})e^{-c N(\log(e+N))^\delta}: N>0\}$ contains the interval $(0,1)$.
Letting $N=N_0$ in \eqref{equ-528-1}, we obtain
\begin{align}\label{equ-528-3}
\|f\|_{L^2(\R^n)} \leq Ce^{CN_0}\|f\|_{L^2(E)}+\varepsilon\|e^{c|\xi|(\log(e+|\xi|))^\delta}\widehat{f}(\xi)\|_{L^2_\xi(\R^n)}.
\end{align}
Since $1+Ce^{CN_0}\lesssim e^{\frac{c}{2} N_0(\log(e+N_0))^\delta}$, we deduce from \eqref{equ-528-2} that
$$
e^{\frac{c}{2} N_0(\log(e+N_0))^\delta}\lesssim\frac{1}{\varepsilon}\lesssim e^{c N_0(\log(e+N_0))^\delta}.
$$
This implies that if $\frac{1}{\varepsilon}>e$,
$$
\ln \frac{1}{\varepsilon}\sim N_0(\log(e+N_0))^\delta,
$$
or equivalently
\begin{align}\label{equ-84-1}
N_0\sim \frac{\log\frac{1}{\varepsilon}}{(\log(e+N_0))^\delta}.
\end{align}
Iterating \eqref{equ-84-1} gives that
\begin{align}\label{equ-819-1}
N_0\lesssim \frac{\log \frac{1}{\varepsilon}}{\left(\log\log \frac{1}{\varepsilon} \right)^\delta}.
\end{align}
We also need the inequality: if $\frac{1}{\varepsilon}>e^e$,
\begin{align}\label{equ-819-2}
C\frac{\log \frac{1}{\varepsilon}}{\left(\log\log \frac{1}{\varepsilon} \right)^\delta}\leq \frac{1-\theta}{\theta}\log\frac{1}{\varepsilon}+e^{C\left(\frac{C\theta}{1-\theta} \right)^{\frac{1}{\delta}}}.
\end{align}
This can be proved by considering the case $C\frac{\log \frac{1}{\varepsilon}}{\left(\log\log \frac{1}{\varepsilon} \right)^\delta}\leq \frac{1-\theta}{\theta}\log\frac{1}{\varepsilon}$ and $C\frac{\log \frac{1}{\varepsilon}}{\left(\log\log \frac{1}{\varepsilon} \right)^\delta}> \frac{1-\theta}{\theta}\log\frac{1}{\varepsilon}$ separately.
Then inserting \eqref{equ-819-1}-\eqref{equ-819-2} into \eqref{equ-528-3}, we find that if $\frac{1}{\varepsilon}>e^e$,
\begin{align}\label{equ-528-4}
\|f\|_{L^2(\R^n)} \leq Ce^{e^{C\left(\frac{C\theta}{1-\theta} \right)^{\frac{1}{\delta}}}}\varepsilon^{-\frac{1-\theta}{\theta}}\|f\|_{L^2(E)}
+\varepsilon\|e^{c|\xi|(\log(e+|\xi|))^\delta}\widehat{f}(\xi)\|_{L^2_\xi(\R^n)}.
\end{align}

Finally, we prove \eqref{equ-84-0}. The proof splits into two cases.

{\bf Case (1). }$\|e^{c|\xi|(\log(e+|\xi|))^\delta}\widehat{f}(\xi)\|_{L^2_\xi(\R^n)}/\|f\|_{L^2(E)}>e^{\frac{e}{\theta}}$.
Choose $\varepsilon$ so that
$$
\varepsilon^{-\frac{1-\theta}{\theta}}\|f\|_{L^2(E)}
=\varepsilon\|e^{c|\xi|(\log(e+|\xi|))^\delta}\widehat{f}(\xi)\|_{L^2_\xi(\R^n)}.
$$
Then
$$
\frac{1}{\varepsilon}=\left(\|e^{c|\xi|(\log(e+|\xi|))^\delta}\widehat{f}(\xi)\|_{L^2_\xi(\R^n)}/\|f\|_{L^2(E)} \right)^\theta> e^e.
$$
We deduce from \eqref{equ-528-4} that
$$
\|f\|_{L^2(\R^n)} \leq Ce^{e^{C\left(\frac{C\theta}{1-\theta} \right)^{\frac{1}{\delta}}}} \|f\|^\theta_{L^2(E)}
  \|e^{c|\xi|(\log(e+|\xi|))^\delta}\widehat{f}(\xi)\|^{1-\theta}_{L^2_\xi(\R^n)}.
$$
Thus \eqref{equ-84-0} holds in this case.

{\bf Case (2). }$\|e^{c|\xi|(\log(e+|\xi|))^\delta}\widehat{f}(\xi)\|_{L^2_\xi(\R^n)}/\|f\|_{L^2(E)}\leq e^{\frac{e}{\theta}}$.
The proof is easier. In fact, using
\begin{align*}
\|f\|_{L^2(\R^n)}\lesssim \|e^{c|\xi|(\log(e+|\xi|))^\delta}\widehat{f}(\xi)\|_{L^2_\xi(\R^n)}\leq e^{\frac{e}{\theta}}\|f\|_{L^2(E)},
\end{align*}
we infer that
\begin{align*}
\|f\|_{L^2(\R^n)}=\|f\|^\theta_{L^2(\R^n)}\|f\|^{1-\theta}_{L^2(\R^n)}\leq \left(Ce^{\frac{e}{\theta}}\|f\|_{L^2(E)}\right)^{\theta}
\left(C\|e^{c|\xi|(\log(e+|\xi|))^\delta}\widehat{f}(\xi)\|_{L^2_\xi(\R^n)}\right)^{1-\theta},
\end{align*}
which implies that \eqref{equ-84-0} also holds.
\end{proof}

\begin{proof}[\textbf{Proof of Theorem \ref{thm-interpolation}}]
Let $E$ be a thick set. It follows from Theorem \ref{cor-un-ana} and Lemma \ref{lem-pre-1} that there exist two constants $C=C(n,E,\sigma)>0$ and $C'=C'(n,E)>0$ so that
\begin{align}\label{equ-84-2}
\int_{\R^n}|f(x)|^2\d x\leq C\left(\int_E|f(x)|^2\d x \right)^\theta\|e^{\sigma|\xi|}\widehat{f}(\xi)\|_{L^2(\R^n)}^{2(1-\theta)}
\end{align}
holds for all $\sigma>0$ and $\theta\in(0,e^{-C'\max\{1,\sigma^{-1}\}})$. Thus Theorem \ref{thm-interpolation} (i) follows from Theorem \ref{thm-ana} (i) and \eqref{equ-84-2}.
Theorem \ref{thm-interpolation} (ii) follows from Theorem \ref{thm-ana} (ii) and Theorem \ref{thm-inter-super}.
\end{proof}

\section{Observability inequalitiesc for fractional heat equations}

In this section, we first present an abstract criterion for some functions to ensure an observability inequality. Then we apply it to solutions of \eqref{frac-heat} and prove Theorem \ref{thm-ob} and Theorem \ref{thm-ob-weight}. We start with the following proposition, which roughly says that an interpolation inequality with a fixed $\theta$ implies an observability inequality.

\begin{proposition}\label{prop-ob-1}
Let $\delta>0$, $0<\theta<1, C\geq 1$ and $E$ be a subset of $\Omega$. Assume that $u(t,x)$ is a function on $[0,1]\times \Omega$ so that $\sup_{0\leq t\leq 1}\limits\|u(t,\cdot)\|_{L^2_x(\Omega)}<\infty$ and
\begin{align}\label{equ-726-1}
\int_\Omega |u(t_2,x)|^2\d x\leq Ce^{\frac{C}{(t_2-t_1)^\delta}}\left(\int_{t_1}^{t_2}\int_E |u(t_2,x)|^2\d x\d t \right)^\theta\left(\int_\Omega |u(t_1,x)|^2\d x\right)^{1-\theta}
\end{align}
holds for all $0<t_1<t_2\leq 1$. Then for every $T\in(0,1]$, we have
\begin{align}\label{equ-726-2}
\int_\Omega |u(T,x)|^2\d x\leq C^{\frac{1}{\theta}}e^{\frac{C+1}{\theta T^\delta}}\int_0^T\int_E |u(t,x)|^2\d x\d t.
\end{align}
\end{proposition}
\begin{proof}
Fix $T>0$. Let $l_1=T$. For every integer $m\geq 2$, define $l_{m}=\lambda^{m-1}l_1$ so that $0<\cdots<l_{m+1}<l_m<\cdots < l_1$ and
\begin{align}\label{equ-726-3}
\frac{l_m-l_{m+1}}{l_{m+1}-l_{m+2}}=\lambda^{-1}:= \left(\frac{C+1}{C+1-\theta}\right)^{\frac{1}{\delta}}>1.
\end{align}
Applying \eqref{equ-726-1} with $t_2=l_{m}$ and $t_1=l_{m+1}$, we have
\begin{align}\label{equ-726-4}
\int_\Omega |u(l_m,x)|^2\d x\leq Ce^{\frac{C}{(l_m-l_{m+1})^\delta}}\left(\int_{l_{m+1}}^{l_{m}}\int_E |u(t,x)|^2\d x\d t \right)^\theta\left(\int_\Omega |u(l_{m+1},x)|^2\d x\right)^{1-\theta}.
\end{align}
Using the inequality $a^\theta b^{1-\theta}\leq \varepsilon^{-(1-\theta)}a+\varepsilon^\theta b \,\;(\forall a,b,\varepsilon>0)$, we deduce from \eqref{equ-726-4} that
$$
\int_\Omega |u(l_m,x)|^2\d x\leq \varepsilon^{-(1-\theta)}C^{\frac{1}{\theta}}e^{\frac{C}{\theta(l_m-l_{m+1})^\delta}}\int_{l_{m+1}}^{l_{m}}\int_E |u(t,x)|^2\d x\d t+ \varepsilon^\theta\int_\Omega |u(l_{m+1},x)|^2\d x
$$
for all $\varepsilon>0$, which can be rewritten  as
\begin{multline}\label{equ-726-5}
\varepsilon^{1-\theta}e^{-\frac{C}{\theta(l_m-l_{m+1})^\delta}}\int_\Omega |u(l_m,x)|^2\d x - \varepsilon e^{-\frac{C}{\theta(l_m-l_{m+1})^\delta}}\int_\Omega |u(l_{m+1},x)|^2\d x\\
\leq  C^{\frac{1}{\theta}}\int_{l_{m+1}}^{l_{m}}\int_E |u(t,x)|^2\d x\d t.
\end{multline}

Letting $\varepsilon=e^{-\frac{1}{\theta(l_m-l_{m+1})^\delta}}$ in \eqref{equ-726-5} and using \eqref{equ-726-3}, we infer that
\begin{multline*}
 e^{-\frac{C+1-\theta}{\theta(l_m-l_{m+1})^\delta}}\int_\Omega |u(l_m,x)|^2\d x -   e^{-\frac{C+1-\theta}{\theta(l_{m+1}-l_{m+2})^\delta}}\int_\Omega |u(l_{m+1},x)|^2\d x\\
\leq  C^{\frac{1}{\theta}}\int_{l_{m+1}}^{l_{m}}\int_E |u(t,x)|^2\d x\d t.
\end{multline*}
Taking the sum over $m\geq 1$, we find
\begin{align}\label{equ-726-6}
 e^{-\frac{C+1-\theta}{\theta(l_1-l_{2})^\delta}}\int_\Omega |u(l_1,x)|^2\d x
&\leq \sum_{m\ge 1} C^{\frac{1}{\theta}}\int_{l_{m+1}}^{l_{m}}\int_E |u(t,x)|^2\d x\d t\nonumber\\
&\leq C^{\frac{1}{\theta}}\int_{0}^{T}\int_E |u(t,x)|^2\d x\d t,
\end{align}
where we used the fact
$$
\lim_{m\to \infty}\limits e^{-\frac{C+1-\theta}{\theta(l_{m+1}-l_{m+2})^\delta}}\int_\Omega |u(l_{m+1},x)|^2\d x=0,
$$
which follows from $\sup_{0\leq t\leq 1}\limits\|u(t,\cdot)\|_{L^2_x(\Omega)}<\infty$ and $l_{m+1}-l_{m+2}\to 0$ as $m\to\infty$. Then \eqref{equ-726-2} follows from \eqref{equ-726-6} clearly.
\end{proof}

We now replace the space-time norm in \eqref{equ-726-1} by a space norm at the final time.

\begin{corollary}\label{cor-ob-1}
Let $\delta>0$, $0<\theta<1, C\geq 1$ and $E$ be a subset of $\Omega$. Assume that $u(t,x)$ is a function on $[0,1]\times \Omega$ so that  for all $t_1<t_2$
\begin{align}\label{equ-727-1}
\int_\Omega |u(t_2,x)|^2\d x\leq C \int_\Omega |u(t_1,x)|^2\d x,
\end{align}
\begin{align}\label{equ-727-2}
\int_\Omega |u(t_2,x)|^2\d x\leq Ce^{\frac{C}{(t_2-t_1)^\delta}}\left(\int_E |u(t_2,x)|^2\d x \right)^\theta\left(\int_\Omega |u(t_1,x)|^2\d x\right)^{1-\theta}.
\end{align}
Then for every $T\in(0,1]$, there exists $C'>0$ depending only on $C,\delta,\theta$ so that
\begin{align}\label{equ-727-3}
\int_\Omega |u(T,x)|^2\d x\leq C'e^{\frac{C'}{T^\delta}}\int_0^T\int_E |u(t,x)|^2\d x\d t.
\end{align}
\end{corollary}
\begin{proof}
Arbitrarily give $0\leq t_1<t_2\leq 1$. Let $s\in(t_1,t_2]$. Applying \eqref{equ-727-2} we have
\begin{align}\label{equ-727-4}
\int_\Omega |u(s,x)|^2\d x\leq Ce^{\frac{C}{(s-t_1)^\delta}}\left(\int_E |u(s,x)|^2\d x \right)^\theta\left(\int_\Omega |u(t_1,x)|^2\d x\right)^{1-\theta}.
\end{align}
Integrating \eqref{equ-727-4} over $s\in [\frac{t_1+t_2}{2},t_2]$, using the H\"{o}lder inequality we infer that
\begin{multline}\label{equ-727-5}
\int_{\frac{t_1+t_2}{2}}^{t_2}\int_\Omega |u(s,x)|^2\d x\d s\\
\leq C(\frac{t_2-t_1}{2})^{1-\theta}e^{\frac{C2^{\delta}}{(t_2-t_1)^\delta}}\left(\int_{\frac{t_1+t_2}{2}}^{t_2}\int_E |u(s,x)|^2\d x \d s \right)^\theta\left(\int_\Omega |u(t_1,x)|^2\d x\right)^{1-\theta}.
\end{multline}
Moreover,  it follows from \eqref{equ-727-1} that
$$
\int_\Omega |u(t_2,x)|^2\d x\leq \frac{2C}{t_2-t_1}\int_{\frac{t_1+t_2}{2}}^{t_2}\int_\Omega |u(s,x)|^2\d x\d s.
$$
This, together with \eqref{equ-727-5}, gives that
\begin{multline}\label{equ-727-6}
\int_\Omega |u(t_2,x)|^2\d x\\
\leq C(\frac{2}{t_2-t_1})^{\theta}e^{\frac{C2^{\delta}}{(t_2-t_1)^\delta}}\left(\int_{\frac{t_1+t_2}{2}}^{t_2}\int_E |u(s,x)|^2\d x \d s \right)^\theta\left(\int_\Omega |u(t_1,x)|^2\d x\right)^{1-\theta}.
\end{multline}
Absorbing $(\frac{2}{t_2-t_1})^{\theta}$ by the exponential term $e^{\frac{C2^{\delta}}{(t_2-t_1)^\delta}}$, and enlarging  the integral interval $[\frac{t_1+t_2}{2},t_2]$ to $[t_1,t_2]$ in \eqref{equ-727-6}, we conclude that
\begin{align}\label{equ-727-7}
\int_\Omega |u(t_2,x)|^2\d x
\leq C_0e^{\frac{C_0}{(t_2-t_1)^\delta}}\left(\int_{t_1}^{t_2}\int_E |u(s,x)|^2\d x \d s \right)^\theta\left(\int_\Omega |u(t_1,x)|^2\d x\right)^{1-\theta}
\end{align}
for some constant $C_0\geq 1$ depending only on $C,\delta,\theta$. Since $t_1<t_2$ can be chosen arbitrarily in \eqref{equ-727-7}, the inequality \eqref{equ-727-3} immediately follows from Proposition \ref{prop-ob-1}.
\end{proof}

With the aid of Corollary \ref{cor-ob-1}, we can prove the observability inequality  for \eqref{frac-heat} by interpolation inequalities as in Theorem \ref{thm-interpolation}.

\begin{proof}[\textbf{Proof of Theorem \ref{thm-ob}.}]
Multiplying \eqref{frac-heat} with $u$ and integrating over $x\in \R^n$, we obtain
$$
\frac{1}{2}\frac{\d}{\d t}\int_{\R^n}|u(t,x)|^2\d x +\int_{\R^n}|\Lambda^{\frac{s}{2}}u(t,x)|^2\d x \leq \int_{\R^n}|a||u(t,x)|^2\d x, \quad \forall t>0.
$$
This, noting $a\in L^\infty(0,\infty; L^\infty(\R^n))$, implies that for some $C>0$
\begin{align}\label{equ-727-8}
\int_{\R^n}|u(t_2,x)|^2\d x\leq Ce^{C(t_2-t_1)}\int_{\R^n}|u(t_1,x)|^2\d x, \quad \forall t_2\geq t_1.
\end{align}

{\bf Case (1). When $0<T\leq 1$.} Let $0\leq t_1<t_2\leq 1$. On one hand, by \eqref{equ-727-8} we have
\begin{align}\label{equ-727-9}
\int_{\R^n}|u(t_2,x)|^2\d x\leq C_1\int_{\R^n}|u(t_1,x)|^2\d x
\end{align}
with $C_1=Ce^C$. On the other hand, it follows from Theorem \ref{thm-interpolation} $(i)$ that
\begin{align}\label{equ-727-10}
\int_\Omega |u(t_2,x)|^2\d x\leq C_2e^{\frac{C_2}{(t_2-t_1)^\delta}}\left(\int_E |u(t_2,x)|^2\d x\d t \right)^\theta\left(\int_\Omega |u(t_1,x)|^2\d x\right)^{1-\theta},
\end{align}
where $\delta=s-1>0$, the constants $C_2\geq 1,\theta\in(0,1)$ depending only on $n,s,$ $a(t,x)$ and $E$. Since \eqref{equ-727-9}-\eqref{equ-727-10} hold for all $t_1<t_2$, according to Corollary \ref{cor-ob-1}, we conclude that
\begin{align}\label{equ-727-11}
\int_\Omega |u(T,x)|^2\d x\leq C_3e^{\frac{C_3}{T^\delta}}\int_0^T\int_E |u(t,x)|^2\d x\d t.
\end{align}
This proves Theorem \ref{thm-ob} in this case.

{\bf Case (2). When $T> 1$.} We first apply \eqref{equ-727-11} with $T=1$ to find that
\begin{align}\label{equ-727-12}
\int_{\R^n} |u(1,x)|^2\d x\leq C_3e^{C_3}\int_0^1\int_E |u(t,x)|^2\d x\d t.
\end{align}
And then we apply \eqref{equ-727-8} with $t_2=T$ and $t_1=1$ to get
\begin{align}\label{equ-727-13}
\int_{\R^n}|u(T,x)|^2\d x\leq Ce^{C(T-1)}\int_{\R^n}|u(1,x)|^2\d x.
\end{align}
Combining \eqref{equ-727-12}-\eqref{equ-727-13} we infer that
$$
\int_{\R^n}|u(T,x)|^2\d x\leq C_3e^{C_3}Ce^{CT}\int_0^T\int_E |u(t,x)|^2\d x\d t.
$$
This shows that Theorem \ref{thm-ob} also holds for $T>1$. It completes the proof.
\end{proof}

To prove Theorem \ref{thm-ob-weight}, we need the following observability inequality for heat equations with lower order terms. Note that it includes a gradient term, so it is a slightly stronger version than Theorem \ref{thm-ob} in the case $s=2$.

\begin{proposition}\label{prop-ob-grad}
Let $E$ be a thick set. Assume that $A$ and $B_1,B_2,\cdots,B_n$ satisfy $\bf (A1)$.
Then there exists $C>0$ depending only on $n,E,A(t,x)$ and ${\bf B(t,x)}=(B_1,B_2,\cdots,B_n)$ so that
$$
\int_{\R^n}|u(T,x)|^2\d x\leq Ce^{C(T+\frac{1}{T})}\int_0^T\int_E|u(t,x)|^2\d x \d t
$$
holds for all $T>0$ and all solutions of
$$
\partial_t u-\Delta u=\mbox{\bf B}(t,x) \cdot \nabla u+A(t,x)u, \quad u(0,x)=u_0(x)\in L^2(\R^n).
$$
\end{proposition}
\begin{proof}
 Since the main idea is the same as the proof of Theorem \ref{thm-ob}, we only give a sketch here. Assume that $A$ and $B_1,B_2,\cdots,B_n$ satisfy $\bf (A1)$ for some $R>0$. Rewrite the equation of $u$ as
\begin{align}\label{equ-729-2}
\partial_t u-\Delta u=\nabla\cdot(\mbox{\bf B}(t,x) u)+\widetilde{A}(t,x)u,\quad u(0,x)=u_0(x)\in L^2(\R^n),
\end{align}
where $\widetilde{A}(t,x)=A(t,x)-\nabla\cdot \mbox{\bf B}$ satisfies $\bf(A1)$ with $R>0$. The integral form of \eqref{equ-729-2} is
\begin{align}\label{equ-729-3}
u(t)=e^{t\Delta}u_0+\int_0^te^{(t-s)\Delta}\Big(\nabla\cdot(\mbox{\bf B}(t,x) u)+\widetilde{A}(t,x)u \Big)\d s.
\end{align}
By Lemma \ref{lem-pre-3}, we have
$$
\|e^{t\Delta}u_0\|_{G^{\sqrt{t}}}\leq C\|u_0\|_{L^2(\R^n)}, \quad \forall t\geq 0.
$$
By Lemma \ref{lem-pre-2}, we have
\begin{align*}
&\sup_{0\leq t\leq T}\left\| \int_0^te^{(t-s)\Delta}\Big(\nabla\cdot(\mbox{\bf B}(t,x) u)+\widetilde{A}(t,x)u \Big)\d s\right\|_{G^{\frac{\sqrt{t}}{2n}}}\\
&\leq C\sup_{0\leq t\leq T} \int_0^t(t-s)^{-\frac{1}{2}}\left\|\Big(\nabla\cdot(\mbox{\bf B}(t,x) u)+\widetilde{A}(t,x)u \Big)\right\|_{G^{\frac{\sqrt{t}}{2n}}}\d s\\
&\leq C\sqrt{t}\sup_{0\leq t\leq T}\|u(t,\cdot)\|_{G^{\frac{\sqrt{t}}{2n}}}, \quad \mbox{ for all  } 0\leq t\leq R^2.
\end{align*}
Combining these two inequalities, using the contraction principle, we infer that there exists a unique solution  $u$ of \eqref{equ-729-3} satisfying
\begin{align}\label{equ-729-4}
 \sup_{0\leq t\leq T}\|u(t,\cdot)\|_{G^{\frac{\sqrt{t}}{2n}}}\leq C\|u_0\|_{L^2(\R^n)}, \quad 0\leq t\leq T_0,
\end{align}
where $T_0>0$ is a small constant. With the bound \eqref{equ-729-4} in hand, similar to the proof of Theorem \ref{thm-ana} (i), after an iteration argument we infer that
\begin{align}\label{equ-729-5}
 \|u(t,\cdot)\|_{G^{R_0}}\leq Ce^{C(t+\frac{1}{t})}\|u_0\|_{L^2(\R^n)}, \quad \forall t\geq 0
\end{align}
with a large constant $C>0$ and some $R_0<R$. Thanks to \eqref{equ-729-5}, it follows from Corollary  \ref{cor-un-ana} that some interpolation inequalities of unique continuation hold. Finally we apply Corollary \ref{cor-ob-1} to conclude the desired observability inequality.
\end{proof}

\begin{proof}[\textbf{Proof of Theorem \ref{thm-ob-weight}.}]
First note that, by changing variable, we can reduce the weighted observability to an unweighted one. In fact, let $u$ be a solution of
$$
\partial_t u -\Delta u=a(t,x)u, \quad u(0,x)=u_0\in L^2\big(\R^n,(1+|x|^2)^\delta\d x\big).
$$
Set $v=(1+|x|^2)^{\frac{\delta}{2}}u$. Then $v$ satisfies the equation
\begin{align}\label{equ-728-1}
\partial_t v -\Delta v=\mbox{\bf B}\cdot \nabla v+ Av, \quad v(0,x)=(1+|x|^2)^{\frac{\delta}{2}}u_0\in L^2\big(\R^n\big),
\end{align}
where
$$
A=\delta(\delta-1)(1+|x|^2)^{-\frac{1}{2}}+a(t,x)
$$
 and
$$
\mbox{\bf B}=(B_1,B_2,\cdots,B_n), \quad B_j=2\delta x_j(1+|x|^2)^{-1}.
$$
Let $E$ be a thick set. If one can show that there exists $C>0$ so that
\begin{align}\label{equ-729-ob}
\int_{\R^n}|v(T,x)|^2\d x\leq Ce^{C(T+\frac{1}{T})}\int_0^T\int_E|v(t,x)|^2\d x \d t
\end{align}
for all $T>0$ and all solutions of \eqref{equ-728-1}, then Theorem \ref{thm-ob-weight} follows clearly.

It remains to show \eqref{equ-729-ob}. Our strategy is to apply Proposition \ref{prop-ob-grad}. To this end, we first claim that, for every $s>0$, the function $h_s(x)=(1+|x|^2)^{-\frac{s}{2}}$ satisfies the bound
\begin{align}\label{equ-729-1}
\|\partial_x^\alpha h_s(x)\|_{L^\infty(\R^n)}\leq C12^{|\alpha|}|\alpha|!, \quad \forall \alpha \in \mathbb{N}^n,
\end{align}
where $C>0$ is a constant depending only on $s,n$. In fact,  according to \cite[Proposition 6.1.5, p.6]{Gra},  the Fourier transform $\widehat{h_s}$ satisfies that
$$
|\widehat{f_s}(\xi)|\leq Ce^{-\frac{1}{2}|\xi|}, \quad |\xi|\geq 2,
$$
and for $|\xi|\leq 2$
$$
|\widehat{f_s}(\xi)|\leq
\left\{
\begin{array}{ll}
C(1+|\xi|^{s-n}), & \quad 0<s<n,\\
C(1+\log\frac{2}{|\xi|}), & \quad s=n,\\
C, & \quad s>n.
\end{array}
\right.
$$
One can check that $e^{\frac{1}{3}|\xi|}\widehat{h_s}(\xi)\in L^2(\R^n)$, and so $h_s\in G^{\frac{1}{3}}$. By \eqref{equ-83-2}, we have
$$
\|\partial_x^\alpha h_s(x)\|_{L^2(\R^n)}\leq C6^{|\alpha|}|\alpha|!, \quad \forall \alpha \in \mathbb{N}^n,
$$
which, together with the Sobolev embedding $H^n(\R^n)\hookrightarrow L^\infty(\R^n)$, proves the claim \eqref{equ-729-1}.

Since $a(t,x)$ satisfies $\bf (A1)$ for some $R>0$, using the bound \eqref{equ-729-1} with $s=1$, we see that $A$ satisfies $\bf(A1)$ with $R'=\min\{R,\frac{1}{4}\}>0$. Moreover, using Leibnitz rule and the bound \eqref{equ-729-1} with $s=2$, we have for all $j=1,2,\cdots,n$
$$
\|\partial_x^\alpha B_j(x)\|_{L^\infty(\R^n)}\leq C12^{|\alpha|}|\alpha|!, \quad \forall \alpha \in \mathbb{N}^n.
$$
Thus every $B_j$ also satisfies $\bf (A1)$. Therefore, we can apply Proposition \ref{prop-ob-grad} to conclude  \eqref{equ-729-ob}. This completes the proof.
\end{proof}

\section*{Acknowledgements}

Wang was partially supported by the National Natural Science Foundation of China under grant No.12171442, and the Fundamental Research Funds for the Central Universities, China University of Geosciences (Wuhan) under grant No.CUGSX01. Zhang  was partially supported by the National Natural Science Foundation of China under grants 11971363, and the Academic Team Building Plan for Young Scholars from Wuhan University under grant 413100085.

\end{document}